\documentclass[a4paper,11pt,reqno,table]{amsart}
\usepackage{amsmath,amsfonts,amsthm}
\usepackage{a4wide}
\usepackage{paralist}

\usepackage{tikz,color,enumitem,hyperref}
\usetikzlibrary{calc,decorations.pathmorphing,intersections,through}

\usepackage{cleveref}
\usepackage{array}
\usepackage{multirow}
\usepackage{algorithm}
\usepackage{algpseudocode}
\usepackage{blkarray, bigstrut} % To do fancy matrices
\usepackage{hhline} % Hack to fix the cell coloring
\usepackage{boldline} % Hack to have boldlines in tables
\usepackage{mathtools}

\theoremstyle{plain}
\newtheorem{theorem}{Theorem}[section]
\newtheorem{thmintro}{Theorem}[section]
\newtheorem{corointro}[thmintro]{Corollary}

\newtheorem*{lemintro*}{Lemma}
\newtheorem*{theorem*}{Theorem}
\newtheorem{lemma}[theorem]{Lemma}
\newtheorem{corollary}[theorem]{Corollary}
\newtheorem{proposition}[theorem]{Proposition}

\theoremstyle{definition}
\newtheorem{definition}[theorem]{Definition}
\newtheorem{question}[theorem]{Question}
\newtheorem{remark}[theorem]{Remark}

\definecolor{darkblue}{rgb}{0,0,0.7} % darkblue color
\definecolor{forestgreen}{rgb}{0.13,0.54,0.13} % forestgreen
\definecolor{palered}{rgb}{1,0.6,0.6} % light red
\definecolor{paleblue}{rgb}{0.8,0.89,1} % light blue
\newcommand{\darkblue}{\color{darkblue}} % darkblue command
\newcommand{\defn}[1]{\emph{\darkblue #1}} % emphasis of a definition

\newcommand\topstrut[1][1.2ex]{\setlength\bigstrutjot{#1}{\bigstrut[t]}}
\newcommand\botstrut[1][0.9ex]{\setlength\bigstrutjot{#1}{\bigstrut[b]}}

\newcommand{\R}{\mathbb{R}}
\newcommand{\Q}{\mathbb{Q}}
\renewcommand{\P}{\mathbb{P}}
\newcommand{\Rbar}{\overline{\R^{d-1}}}

\DeclareMathOperator{\conv}{conv}
\DeclareMathOperator{\inter}{int}
\DeclareMathOperator{\gl}{GL}
\DeclareMathOperator{\vertices}{Vert}
\DeclareMathOperator{\aff}{aff}

\definecolor{orange}{rgb}{0.898, 0.621, 0.0}
\definecolor{skyblue}{rgb}{0.336, 0.703, 0.910}
\definecolor{bluishgreen}{rgb}{0, 0.617, 0.449}
\definecolor{yellow}{rgb}{0.937, 0.890, 0.258}
\definecolor{blue}{rgb}{0, 0.445, 0.695}
\definecolor{red}{rgb}{0.832, 0.367, 0}
\definecolor{purple}{rgb}{0.797, 0.473, 0.652}

\newcommand{\wddots}{{\textcolor{white}{$\ddots$}}}
\newcommand{\bcheck}{\cellcolor{paleblue} $\checkmark$}
\newcommand{\bddots}{\cellcolor{paleblue} $\ddots$}
\newcommand{\bcdots}{\cellcolor{paleblue} $\cdots$}
\newcommand{\rcdots}{\cellcolor{palered} $\cdots$}
\newcommand{\rtimes}{\cellcolor{palered} $\times$}
\newcommand{\rddots}{\cellcolor{palered} $\ddots$}
\newcommand{\lnum}[1]{\rule{-2mm}{0mm}{#1}}
\newcommand{\lcel}[1]{\rule{-6mm}{0mm}{#1}}

\newcommand{\tf}[1]{$\begin{matrix*}[r]#1\end{matrix*}$}

\newcommand{\bn}{$\!\!\!\!\!$}
\newcommand{\tcw}[1]{\textcolor{white}{#1}}
\newlength{\Oldarrayrulewidth}

\newcolumntype{?}[1]{!{\vrule width #1}}

\title[Combinatorial Inscribability Obstructions for Higher-Dimensional Polytopes]{Combinatorial Inscribability Obstructions\\for Higher-Dimensional Polytopes}

\date{\today}

\author[J. Doolittle, J.-P. Labb\'e, C.E.M.C. Lange, R. Sinn, J. Spreer, and G.M. Ziegler]{Joseph Doolittle
        \and
	Jean-Philippe Labb\'e
        \and
        Carsten E. M. C. Lange
	\and
	Rainer Sinn
	\and
	Jonathan Spreer
	\and
	G\"unter M. Ziegler}

\address[J. Doolittle, J.-P. Labb\'e, R. Sinn, G.M. Ziegler]{Institut f\"ur Mathematik, Freie Universit\"at Berlin, Arnimallee 2, 14195 Berlin, Germany}
\email{$\{$jdoolitt,labbe,rsinn,ziegler$\}$@math.fu-berlin.de,}
\urladdr{http://page.mi.fu-berlin.de/$\{$jdoolitt,labbe,rsinn,gmziegler$\}$}

\address[C.E.M.C.~Lange]{Fakult\"at f\"ur Mathematik, Technische Universit\"at M\"unchen, D-85748, Garching, Germany}
\email{lange@ma.tum.de}
\urladdr{https://geo.ma.tum.de/de/personen/carsten-lange.html}

\address[J. Spreer]{School of Mathematics and Statistics F07, University of Sydney, NSW 2006 Australia}
\email{jonathan.spreer@sydney.edu.au}
\urladdr{https://www.tacet.de/Jonathan/}

\thanks{This work was supported by the DFG Collaborative Research Center TRR~109 ``Discretization in Geometry and Dynamics.''}

\keywords{cyclic polytope, inscribability, circumscribability, $f$-vector}
\subjclass[2010]{Primary 52B11; Secondary 51M20, 52B05}

\begin{document}

\begin{abstract} 
For $3$-dimensional convex polytopes, inscribability is a classical property that is relatively well-understood due to its relation with Delaunay subdivisions of the plane and hyperbolic geometry. 
In particular, inscribability can be tested in polynomial time, and for every $f$-vector of $3$-polytopes, there exists an inscribable polytope with that $f$-vector.
For higher-dimensional polytopes, much less is known.
Of course, for any inscribable polytope, all of its lower-dimensional faces need to be inscribable, but this condition does not appear to be very strong.

We observe non-trivial new obstructions to the inscribability of polytopes that arise when imposing that a certain inscribable face be inscribed. Using this obstruction, we show that the duals of the $4$-dimensional cyclic polytopes with at least $8$ vertices---all of whose faces are inscribable---are not inscribable. This result is optimal in the following sense: We prove that the duals of the cyclic $4$-polytopes with up to $7$ vertices are, in fact, inscribable.

Moreover, we interpret this obstruction combinatorially as a forbidden subposet of the face lattice of a polytope, show that $d$-dimensional cyclic polytopes with at least $d+4$ vertices are not circumscribable, and that no dual of a neighborly $4$-polytope with $8$ vertices, that is, no polytope with $f$-vector $(20,40,28,8)$, is inscribable.
\end{abstract}

\maketitle

\section{Introduction and background}

The convex hull of a finite number of points on a sphere is an \defn{inscribed polytope}.
Choosing the points randomly on the sphere almost surely gives a simplicial polytope.
However, choosing these points carefully, one may obtain other types of polytopes.
In 1832, Steiner asked whether it is possible to obtain every $3$-dimensional polytope this way \cite[Question 77, p.~316]{steiner_systematische_1832}.
A polytope is \defn{inscribable} if it is combinatorially equivalent to an inscribed polytope, i.e., if it has a realization that is inscribed.
Around 100 years later, Steinitz provided the first examples of polytopes that are not inscribable \cite{steinitz_ueber_1928}.
Such polytopes without an inscribed realization include the simplicial polytope obtained by stacking each triangle of the tetrahedron, see \cite[Section 13.5]{gruenbaum_convex_1967} and \cite[p.~140]{steinitz_ueber_1928}. 
In light of this, one may ask to what extent a combinatorial property of a polytope (simplicity, simpliciality, neighborlyness, stackedness, etc.) can restrict its inscribability.
Gonska and Ziegler asked whether inscribable polytopes affect a coarser polytope invariant, the $f$-vector \cite[Introduction]{gonska_inscribable_2013}.
Indeed, experimental results seem to indicate that sufficient conditions for inscribability may be obtained from the $f$-vector \cite[Section~2]{padrol_topics_2016}.
For more detail on related questions and their history, we refer to the recent articles \cite{gonska_inscribable_2013,padrol_topics_2016,chen_scribability_2017} and references therein. 

Due to its inherent relation with \emph{Delaunay tesselations} \cite{brown_voronoi_1979} and planar $3$-connected graphs \cite{steinitz_ueber_1928,gruenbaum_some_1987,dillencourt_graph_1996}, inscribability of $3$-dimensional polytopes has garnered attention and consequently is relatively well understood.
Hodgson, Rivin and Smith following work by Rivin use hyperbolic geometry to show that a $3$-polytope is inscribable if and only if a certain system of linear inequalities has a solution, \cite{hodgson_characterization_1992,rivin_characterization_1996}.
Similar to other problems in polytope theory (e.g. characterization of $f$-vectors or of vertex-edge graphs), the methods of Hodgson, Rivin and Smith
do not extend to higher dimensions and relatively little is known for $d$-dimensional polytopes (or \defn{$d$-polytopes}).
Numerous classes of polytopes have been determined to be inscribable.
Among them are the cyclic $d$-polytopes, see \cite[Section~2.5.2]{gonska_inscribable_2013} for three proofs.
Gonska and Ziegler provide a strikingly simple combinatorial characterization of inscribable stacked polytopes: a stacked polytope is inscribable if and only if all nodes of its dual tree have degree at most three, \cite[Theorem 1]{gonska_inscribable_2013}.
Earlier, graph-theoretical necessary conditions and sufficient conditions for a $3$-polytope to be inscribable were provided by Steinitz \cite{steinitz_ueber_1928} and \cite[Section~13.5]{gruenbaum_convex_1967} as  well as Dillencourt and Smith \cite{dillencourt_graph_1996}.
Of course, every face of an inscribed polytope must be inscribed, so the inscribability conditions of $3$-polytopes impose natural conditions on higher dimensional polytopes, see e.g. \cite[Section~12]{rivin_characterization_1996} and \cite[Section~2]{padrol_topics_2016}.
In particular, the conditions can be used as a first check to determine the non-inscribability of some polytopes in dimension~$4$.
For simplicial $4$-polytopes with at most $10$ vertices, Firsching combined these results with nonlinear optimization to determine inscribability of all but 13 types \cite[Theorem~25]{firsching_realizability_2017}.
This shows that even small polytopes can satisfy the necessary conditions but may fail to have an obvious inscribed realization. 
In which case, new efficient methods have to be developed to determine the inscribability of combinatorial types of polytopes in higher dimension \cite[Question~3]{firsching_realizability_2017}.

In this article, we study the inscribability of higher-dimensional polytopes and describe an obstruction to inscribability using face lattices of polytopes.
We provide an approach to studying inscribability that makes use of higher-dimensional facial incidence information, in contrast to using only the graph of the polytope.
Namely, we present ``Miquel's polytopes'', a class of $3$-polytopes steming from Miquel's circle theorem used in the following lemma.

\begin{lemintro*}[Obstruction Lemma]
If a polytope $P$ has a Miquel polytope $M$ as a $3$-face with a prescribed incidence relation with another vertex of $P$, then $P$ has no realization such that~$M$ is inscribed. 
\end{lemintro*}

As a direct consequence of this lemma, we answer several questions related to inscribability.
For instance, Miquel's polytopes with this incidence relationship are found in dual to cyclic polytopes.
\begin{thmintro}[\Cref{thm:no_real_c48}]
\label{thm:A}
Let $k\geq 8$. No realization of $C_4(k)^*$ has an inscribed facet, although all its facets are inscribable. 
\end{thmintro}
Chen and Padrol proved that $C_d(k)^*$ is not inscribable provided $k$ is large enough \cite[Theorem~2]{chen_scribability_2017}. 
They were able to provide a super-exponential bound in \(d\) for $k$ that guarantees non-inscribability. 
Extending the argumentation of \Cref{thm:A} leads to an effective bound on the non-inscribability of the duals of cyclic polytopes.
\begin{corointro}[\Cref{cor:all_cyclic}]
The dual of the $d$-dimensional cyclic polytope on $k$ vertices $C_d(k)^*$ is inscribable if
\begin{center}
\begin{tabular}{c@{,\quad or\quad}c@{,\quad or\quad}c}
$d\leq 3$ & $d=4$ and $k=7$ & $k\leq d+2$.
\end{tabular}
\end{center}
If $k\geq d+4\geq 8$, then $C_d(k)^*$ is not inscribable.
\end{corointro}
Thus the only class of dual to cyclic polytopes whose inscribability is not determined is \(C_d(d+3)^*\) for $d\geq 5$: Are they inscribable?
This seems to be a challenging problem.
For a summary of the results on cyclic polytopes, see the discussion at the end of \Cref{sec:cyclic_poly} and \Cref{tab:cyclic_poly}.
Further, we provide some evidence in support of 
\cite[Conjecture~8.4]{chen_scribability_2017} 
% CAREFUL: In the arXiv preprint this is Conjecture 7.4, in the published version it is Conjecture 8.4
that neighborly polytopes with sufficiently many vertices are not circumscribable.
\begin{thmintro}[\Cref{thm:n48}]
If a polytope is dual to a neighborly $4$-polytope on $8$ vertices, then that polytope is not inscribed.
\end{thmintro}

We denote the \defn{$f$-vector} of a $d$-polytope $P$ by $f_P = (f_0,f_1,\dots, f_{d-1})$, where $f_i$ is the number of $i$-dimensional faces of $P$.
An $f$-vector is \defn{inscribable} if at least one polytope with that $f$-vector is inscribable.
Gonska and Ziegler provide a combinatorial characterization of inscribable stacked polytopes. 

As all stacked $d$-polytopes with $n$ vertices have the same $f$-vector, 
their results imply that for $d+1\le n\le d+4$ all stacked polytopes are inscribable, while
for $n \geq d+5 \geq 8$ there exist inscribable as well as not inscribable stacked polytopes.

The influence of inscribability on $f$-vectors remained elusive.

As the $f$-vector is a coarse polytope invariant, there can be huge numbers of different combinatorial types of polytopes for a given $f$-vector. 
Therefore, to determine whether a specific given $f$-vector is inscribable, we use known classifications of combinatorial types of polytopes.
For example, there are three combinatorial types of $4$-polytopes that share the $f$-vector $(20,40,28,8)$,
which corresponds to the duals of the neighborly $4$-polytopes on $8$ vertices.
The dual to the cyclic polytope on $8$ vertices, $C_4(8)^*$, is the most prominent example.%

As a corollary of the above theorem, we exhibit the first $f$-vector that is not inscribable. 
This result provides the first evidence towards an answer to the question 
\emph{How does the condition of inscribability restrict the f-vectors of polytopes?}
raised in \cite[Introduction]{gonska_inscribable_2013}.
\begin{corointro}
\label{cor:main}
The $f$-vector $(20,40,28,8)$ is not inscribable.
\end{corointro}

Beyond the previous considerations, we emphasize three notable aspects of the Obstruction Lemma.
Starting in dimension $4$, there are polytopes such that every facet is inscribable but \emph{no realization} of the polytope has \emph{any} inscribed facet.
Previous conditions on inscribability of polytopes were derived from their graphs.
This obstruction is \emph{different}: it uses higher-dimensional facial incidences, and it may be used to obtain obstructions in arbitrary face-figures.
This makes it a flexible combinatorial tool to obstruct inscribability.
Finally, the obstruction comes from a rather unrestrictive forbidden subposet and \emph{appears naturally in many common} $4$-polytopes. 
Out of the 1294 \(4\)-polytopes with \(8\) facets, 169 of them have a Miquel polytope as a facet. 
Of these 169 \(4\)-polytopes, twenty of them also have the required incidence relations to guarantee non-inscribability.
\medskip

\noindent
\textbf{Outline.}
In \Cref{sec:small_inscribed}, we study inscribability of $f$-vectors of polytopes with few vertices and facets.
In \Cref{sec:cyclic_poly}, we examine the inscribability of duals of cyclic polytopes and prove that ``most'' of these polytopes are not inscribable.
In \Cref{sec:forbidden}, we present the combinatorial obstruction to inscribability in terms of a forbidden subposet.
In \Cref{sec:neighborly}, we extend the obstruction to neighborly $4$-polytopes with $8$ vertices.
In \Cref{sec:questions}, we present three questions that arose during our investigation of inscribed polytopes.
\medskip

\noindent
\textbf{Acknowledgements.} 
The authors are grateful to Hao Chen, Alexander Fairley, Moritz Firsching, Arnau Padrol, G\"unter Rote, Francisco Santos, and Raman Sanyal for valuable discussions.

%%%%%%%%%%%%%%%%%%%%%
\section{Inscribability and small $f$-vectors}
\label{sec:small_inscribed}
%%%%%%%%%%%%%%%%%%%%%

In this section, we set the context surrounding the inscribability of polytopes and $f$-vectors.
For basic polytope nomenclature and constructions, we refer the reader to \cite{ziegler_lectures_1995,henk_basic_2018}.

%%%%%%%%%%%%%%%%%%%%%
\subsection{Inscribability and stereographic projections}
\label{ssec:background}
%%%%%%%%%%%%%%%%%%%%%
Alternatively to putting vertices of a polytope on a sphere, one may ask that all of its supporting hyperplanes be tangent to the sphere, in which case we say that the polytope is \defn{circumscribed}.
Similarly to inscribability, a polytope is \defn{circumscribable} if it has a realization that is circumscribed.
As Steinitz first observed \cite{steinitz_ueber_1928}, inscribability and circumscribability are notions related by polytope duality: a polytope is inscribable if and only if its dual is circumscribable. 
Hence, every statement about inscribability has an equivalent formulation in terms of circumscribability and we implicitly make use of this fact throughout the text.
We collect classical results on inscribability and circumscribability in the next two lemmas.

\begin{lemma}
\label{lem:classic}
Let $P$ be a $d$-polytope with vertex $v$ and $P^*$ be its dual.
\begin{enumerate}[label=\roman{enumi}),ref=\ref{lem:classic}~\emph{\roman{enumi})}]
	\item $P$ is inscribable if and only if $P^*$ is circumscribable, see e.g. \cite[Theorem~13.5.1]{gruenbaum_convex_1967}.\label{lem:classic_i}
	\item If $P$ is circumscribable, then so is the vertex figure of $v$ in $P$. \label{lem:classic_ii}
	\item If $P$ is inscribable, then so are its faces. \label{lem:classic_iii}
\end{enumerate}
\end{lemma}

Let $\mathbb{S}^{d-1}\subset \R^d$ denote the $(d-1)$-dimensional sphere of radius $\frac{1}{2}$ centered at $e_d:=(0,\dots,0,\frac{1}{2})$.
The points $N:=(0,\dots,0,1), S:=(0,\dots,0)\in\mathbb{S}^{d-1}$ are the North and South Pole of $\mathbb{S}^{d-1}$.
Moreover, we denote the one point compactification of $\R^{d-1}=\R^{d-1}\times \{0\}\subset \R^d$ by $\Rbar:=\R^{d-1} \cup \{\infty\}$.
The stereographic projection 
\[
	\pi_N:\mathbb{S}^{d-1}\rightarrow\Rbar
\]
from the point $N$ maps $x\in \mathbb{S}^{d-1}\setminus\{N\}$ to the intersection of the line through $x$ and~$N$ with $\R^{d-1}$, and~$N$ to $\infty$.
Let $P$ be a $d$-polytope with vertex $v$ and $H$ be a hyperplane that strictly separates \(v\) from \(\vertices(P) \setminus \{v\}\).
Then 
\[
	\pi_v:P\setminus\{v\} \rightarrow H
\]
denotes the \defn{stereographic projection of $P$ from $v$} defined analogously to the stereographic projection $\pi_N$.
If $P$ is inscribed on $\mathbb{S}^{d-1}$ and $v$ is rotated to $N$, then the two projections map $\mathbb{S}^{d-1}\cap P\setminus\{v\}$ to projectively equivalent labeled sets.

\begin{lemma}\label{lem:projection}
Let $P$ be a $d$-polytope, and $v$ be a vertex of $P$ contained in exactly $d$ facets. 
The stereographic projection $\pi_{v}$ yields the following structures.

\begin{enumerate}[label=\roman{enumi}),ref=\ref{lem:projection}~\emph{\roman{enumi})}]
\item The images of facets of $P$ that contain $v$ bound a $(d-1)$-dimensional simplex $\Delta$. \label{lem:projection_i}
\item The images of the vertices of \(P\) determine a point configuration such that the images of the faces of \(P\) that do not contain \(v\) form a polytopal subdivision of \(\Delta\).\label{lem:projection_iii}
\item The images of facets of $P$ that do not contain $v$ are $(d-1)$-dimensional polytopes. \label{lem:projection_ii}
\item If $P$ is inscribed, then the images of facets of \(P\) that do not contain \(v\) are inscribed. \label{lem:projection_iv}
\end{enumerate}
\end{lemma}

\begin{proof}
i) The projection of $P$ from \(v\) yields the vertex figure $P/v$, see \cite[Proposition~2.4]{ziegler_lectures_1995}.

ii) The projection~$\pi_v$ acts on faces of \(P\) that do not contain \(v\) as an affine map from \(\R^d\) to \(\R^{d-1}\). 
The polytopal complex of the faces of \(P\) that do not contain \(v\) is preserved by this affine map. 
By part \emph{i)}, the union of the images of these faces is \(\Delta\). 
This satisfies the definition of a polyhedral subdivision, see \cite[Definition~2.3.1 and Lemma~4.2.20]{de_loera_triangulations_2010}.

iii) The affine span of a facet of \(P\) that does not contain \(v\) does not intersect \(v\). 
Consequently, the projection of such a facet under the affine map \(\pi_v\) preserves the facet's dimension.

iv) Suppose \(P\) is inscribed and \(F\) is a facet of \(P\) that does not contain \(v\). 
Let \(S\) be the intersection of \(\aff(F)\) with the sphere inscribing \(P\). 
By \emph{iii)}, the image of \(F\) is a polytope whose vertices lie on the image of \(S\). 
The image of \(S\) is a \((d-2)\)-dimensional sphere.
See the related discussion in \cite[Section 2.4]{gonska_inscribable_2013}. 
\end{proof}

%%%%%%%%%%%%%%%%%%%%%
\subsection{Inscribed realizations of small $f$-vectors}
%%%%%%%%%%%%%%%%%%%%%

To study inscribability in dimension larger than $3$, we need small examples of not inscribable polytopes to contrast with the inscribable ones. 
A natural place to look for small examples is among $4$-polytopes with small \(f\)-vector. 
Often, a $d$-polytope $P$ (or its $f$-vector) is considered small if $f_0$ (or dually, $f_3$) is small. 
Another natural measure for the size of a $4$-polytope~$P$ or its $f$-vector is the the sum $f_0+f_3$.
This number counts the vertices of the vertex-facet adjacency graph that determines the combinatorial type of~$P$.
One of our motivating questions is:
\begin{center}
\emph{Is there an $f$-vector that is not inscribable?}
\end{center}
If such an $f$-vector exists, the following question is natural:
\begin{center}
\emph{What is the smallest $f$-vector that is not inscribable?}
\end{center}
In \Cref{sec:cyclic_poly,sec:neighborly} we show that such an $f$-vector indeed exists and provide the first example of an $f$-vector that is not inscribable. 
As many combinatorially distinct $d$-polytopes can have the same $f$-vector, an $f$-vector is not inscribable if \emph{every} polytope with this $f$-vector is \emph{not} inscribable.
Firsching \cite{firsching_complete_2018} extended previous classifications of $4$-polytopes with few vertices by Altshuler and Steinberg \cite{altshuler_complete_1985} and Brinkmann \cite{brinkmann_fvector_2016}.
For a thorough historical account, we refer to \cite[Section~1.4]{firsching_complete_2018} and the references therein.
A complete enumeration of all $4$-polytopes with $f_0\leq 9$ or $f_3\leq 9$ exists and partial results are known for $f_0,f_3\geq 10$ and $20 \leq f_0+f_3\leq 23$.
\Cref{tab:fvector_f0_is_7}--\Cref{tab:fvector_f0_is_9} list the total number of $f$-vectors and of combinatorial types for all possible pairs $(f_0,f_3)$ with $7 \leq f_3 \leq 9$.
The Euler--Poincar\'e formula determines $f_1$ from $f_0$, $f_2$ and $f_3$.
For a given a value of \(f_3\) and \(f_0\), for each possible value of \(f_2\), we write the pair \(f_2:N\), where \(N\) is the number $N$ of combinatorially distinct $4$-polytopes with the specified \(f_0, f_2, f_3\).

We discuss the inscribability of small \(f\)-vectors derived from the these enumeration results.
The only $4$-polytope with $5$ vertices or $5$ facets is the simplex which is clearly inscribable. 
If $f_0=6$, then $6\leq f_3\leq 9$ and each of the four pairs of $(f_0,f_3)$ determine a unique $4$-polytope. 
If $f_0=7$, then $6\leq f_3\leq 14$ and there are $15$ distinct $f$-vectors and $31$ combinatorially distinct polytopes. 
All these $35$ polytopes are inscribable and inscribing vertex-coordinates are provided in \Cref{appendix:inscribed_coordinates}.

For $f_0=8$ there are $40$ distinct $f$-vectors and $1294$ combinatorially distinct polytopes and for $f_0=9$ there are $88$ disctinct $f$-vectors and $274\ 148$ distinct polytopes.
Only $8$ out of these $128$ $f$-vectors with $8\leq f_0 \leq 9$ determine a combinatorially unique polytope.

%%%%%%%%%%%%%%%%%%%%
%% Table f_0=7
%%%%%%%%%%%%%%%%%%%%
\begin{table}[!b]
\resizebox{0.75\textwidth}{!}{
\rule{2cm}{0cm}
\begin{tabular}{r|r|r|r|r|r|r|r|r|r}
%  	  				 &				&			 &			  &			    &				& $f_3$ 	\\
$f_0$  			  	  & 6			& 7 	   &  8		  &  9		 &  10		&  11		&  12		&  13		&  14\tcw{$\ \star$}  		\\ \hline
\# of $(f_0,*,*,7)$	  & 1			& 2		   &  2		  &  2	 	 &   2		&   2		&   2		&   1		&   1\tcw{$\ \star$}       \\ 
\# of combin. types	  & 1			& 3		   &  5		  &  7		 &   6		&   4		&   3		&   1		&   1\tcw{$\ \star$}		\\ \hline
\tf{f_2\!:&\bn\text{\# combin. types}\\
		 &	}	  
		 &   \tf{15\!:&\bn 1\\		% 6
		   \tcw{41\!:}&\bn\tcw{1}}
						 & \tf{16\!:&\bn 2\\	% 7
							   17\!:&\bn 1}
							 		& \tf{17\!:&\bn 4\\	% 8
										  18\!:&\bn 1}
											   & \tf{17\!:&\bn 1\\	% 9
													 18\!:&\bn 6}
							 							  & \tf{18\!:&\bn 4\\		%10
			   								  				   	19\!:&\bn 2}
							 							   			 & \tf{18\!:&\bn 1\\			% 11
													 					   19\!:&\bn 3}
							 							   						& \tf{19\!:&\bn  2\\			% 12
						 			   								  				  20\!:&\bn  1}
						 							 							   		   & \tf{20\!:&\bn  1\\
			   								   								   			   \tcw{41\!:}&\bn\tcw{1}}
							 							   								   			  & \tf{\star\ 21\!:&\bn   1\ \star\\			% 14
																									  		 \tcw{41\!:}&\bn\tcw{1}}
\end{tabular}
\rule{2cm}{0cm}
}
\vspace{0.5cm}
\caption{Complete enumeration of $f$-vectors of $4$-polytopes with $f_3=7$.}
\label{tab:fvector_f0_is_7}

\end{table}

Since no efficient algorithm is known to decide inscribability for $d$-polytopes with $d\geq 4$, 
a natural heuristic to find an $f$-vector that is not inscribable is to study inscribability for small $f$-vectors associated to a combinatorially unique polytope. 
We briefly indicate results obtained by this search.
 
%%%%%%%%%%%%%%%%%%%%
%% Table f_0=8
%%%%%%%%%%%%%%%%%%%%
\begin{table}[!t]
\resizebox{\textwidth}{!}{
\begin{tabular}{r|r|r|r|r|r|r|r|r|r|r|r|r|r|r|r}
%  	  				 &				&			 &			  &			    &				& $f_3$ 	\\
$f_0$  			  	  & 6			& 7 			&  8			&  9\tcw{$\ \star$}			&  10			&  11		&  12		&  13		&  14  		&  15		& 16	  &  17		  &  18  	 &  19		 & 20\tcw{$\ \star$}  \\ \hline
\# of $(f_3,*,*,8)$	  & 1			& 2				&  3			&  4\tcw{$\ \star$}	 		&   4			&   3		&   4		&   3		&   4       &   3		&  3	  &   2		  &   2       &   1		  &  1\tcw{$\ \star$}  \\ 
\# of combin. types	  & 1			& 5				& 27			& 76\tcw{$\ \star$}			& 137			& 205		& 225		& 218		& 166		& 117		& 65		&  31		&  14		&   4		&  3\tcw{$\ \star$}  \\ \hline
\tf{f_2\!:&\bn\text{\# combin. types}\\
		 & 					  \\
		 &					  \\
		 & }	  & \tf{16\!:&\bn 1\\		% 6
							 &\bn \\
			     \tcw{241\!:}&\bn\tcw{1}\\
							 &\bn  }
							 	& \tf{18\!:&\bn 4\\			% 7
							 		  19\!:&\bn 1\\
							   \tcw{241\!:}&\bn\tcw{1}\\
							   			   &\bn  }
							 			   		& \tf{19\!:&\bn 13\\	% 8
													  20\!:&\bn 12\\
													  21\!:&\bn  2\\
			   								   	\tcw{241\!:}&\bn\tcw{1}}
																& \tf{19\!:&\bn  1\\	% 9
													  		  		  20\!:&\bn 31\\
													 				  21\!:&\bn 37\\
			   														  22\!:&\bn  7}
							 							   			  		& \tf{20\!:&\bn  7\\		%10
			   								  				   			   		  21\!:&\bn 71\\
											   			   				   		  22\!:&\bn 56\\
																				  23\!:&\bn  3}
							 							   				   			   			& \tf{21\!:&\bn  26\\			% 11
													 						   		  		  		  22\!:&\bn 128\\
																									  23\!:&\bn  51\\
			   								   								   				   \tcw{41\!:}&\bn\tcw{1}}
							 							   								   				   & \tf{21\!:&\bn   4\\			% 12
						 			   								  				   			   		  		 22\!:&\bn  75\\
						 											   			   				   		  	   	 23\!:&\bn 129\\
						 																				  	     24\!:&\bn  17}
						 							 							   				   			   			& \tf{22\!:&\bn  16\\			% 13
																															  23\!:&\bn 112\\
																															  24\!:&\bn  90\\
			   								   								   				   						   \tcw{241\!:}&\bn\tcw{1}}
							 							   								   				   							& \tf{22\!:&\bn   3\\			% 14
																																		  23\!:&\bn  30\\
																																		  24\!:&\bn 103\\
																																		  25\!:&\bn  30}
						 							 							   				   			   							  		& \tf{23\!:&\bn   5\\			% 15
																																					  24\!:&\bn  39\\
																																					  25\!:&\bn  73\\
																																				\tcw{241\!:}&\bn\tcw{1}}
							 							   								   				   												&  \tf{24\!:&\bn   8\\			% 16
																																							   25\!:&\bn  32\\
																																							   26\!:&\bn  25\\
																																						\tcw{241\!:}&\bn\tcw{1}}
					 							 							   															   			   			& \tf{25\!:&\bn  8\\			% 17
																																									  26\!:&\bn 23\\
																																								\tcw{241\!:}&\bn\tcw{1}\\
																																								   		   &\bn }  
 							 							   																							   							& \tf{26\!:&\bn  6\\			% 18
 																																													  27\!:&\bn  8\\
 																																												\tcw{241\!:}&\bn\tcw{1}\\
 																																												   		   &\bn}
 						 							 							   															   			   							  		& \tf{27\!:&\bn   4\\			% 19
 																																														   \tcw{241\!:}&\bn\tcw{1}\\
 																																														   			   &\bn\\
 																																																	   &\bn}
 							 							   																							   												&  \tf{\star\ 28\!:&\bn   3\ \star\\			% 20
 																																																	       \tcw{241\!:}&\bn\tcw{1}\\
 																																														   			   		 		   &\bn\\
 																																																	   		 		   &\bn}
\end{tabular}
}
\vspace{0.5cm}
\caption{Complete enumeration of $f$-vectors of $4$-polytopes with $f_3=8$.}
\label{tab:fvector_f0_is_8}

\end{table}

%%%%%%%%%%%%%%%%%%%%
%% Table f_0=9
%%%%%%%%%%%%%%%%%%%%
\begin{table}[!b]
\resizebox{\textwidth}{!}{
\begin{tabular}{r|r|r|r|r|r|r|r|r|r|r|r}
%  	  				 &				&			 &			  &			    &				& $f_3$ 	\\
$f_0$  				  & 6        	& 7 				& 8         	& 9\tcw{$\ \star$}          & 10		  		& 11			&  12			& 13		& 14  			& 15			& 16	\\ \hline
\# of $(f_0,*,*,9)$	  & 1			& 2					& 4		  		& 6\tcw{$\ \star$}	 		& 5			  		& 6				& 6			 	& 6			& 5	        	& 6				& 5		\\ 
\# of combin. types	  & 1			& 7					& 76			& 463\tcw{$\ \star$}		& 1872		  		& 5218			& 11277		 	& 19666		& 28821			& 36105			& 39436	\\ \hline
\tf{f_2\!:&\bn\text{\# combin. types}\\
		 & 					  \\
		 &					  \\
		 &					  \\
		 &					  \\
		 & }	  & \tf{18\!:&\bn 1\\
							 &\bn \\
							 &\bn  \\
							 &\bn  \\
			   \tcw{241\!:}&\bn\tcw{1}\\
							 &\bn  }
							 	& \tf{19\!: &\bn 1\\
									  20\!: &\bn 6\\
								 	 		&\bn  \\
									   	 	&\bn  \\
							  \tcw{241\!:}&\bn\tcw{1} \\
									    	&\bn}	& \tf{20\!: & \bn 1\\
											  		      21\!: & \bn 31   \\
													      22\!: & \bn 37   \\
													      23\!: & \bn  7   \\
												  \tcw{241\!:}&\bn\tcw{37}\\
														   	    & \bn}	
																	& \tf{20\!: & \bn   1\tcw{\ \star}\\
														   				  22\!: & \bn 129\tcw{\ \star}\\
																		  23\!: & \bn 209\tcw{\ \star}\\
																  \tcw{19}24\!: & \bn 116\tcw{\ \star}\\
																		  25\!: & \bn   7\tcw{\ \star}\\
																          26\!: & \bn   1\ } 
																		  		  & \tf{22\!: &\bn  12\\
																				  		   23\!: &\bn 397\\
																				  		   24\!: &\bn 897\\
																				  		   25\!: &\bn 504\\
																				   \tcw{9}26\!: &\bn  62\\
																				      	 	   	 &} & \tf{23\!: &\bn   65\\
																					  			  		  24\!: &\bn 1185\\
																					  			  	 	  25\!: &\bn 2593\\
																								  \tcw{19}26\!: &\bn 1266\\
																								 	  	  27\!: &\bn  107\\
																							 		  	  28\!: &\bn    2} 
																											  		& \tf{23\!: & \bn   3\\
																														  24\!: & \bn 333\\
																														  25\!: & \bn3250\\
																												   \tcw{9}26\!: & \bn5662\\
																												   		  27\!: & \bn1943\\
																														  28\!: & \bn  86} 
																																	& \tf{24\!: &\bn   33\\
																															  		  	  25\!: &\bn 1219\\
																																	  	  26\!: &\bn 7536\\
																															   	   \tcw{9}27\!: &\bn 9023\\
																																		  28\!: &\bn 1829\\
																																		  29\!: &\bn   26}
																																	  	  		& \tf{25\!: &\bn   205\\
																																					  26\!: &\bn  3608\\
																																					  27\!: &\bn 13744\\
																																					  28\!: &\bn 10268\\
																																				\tcw{2}29\!: &\bn   966\\
																																					 	  	&		}
																																								& \tf{25\!: &\bn    15\\
																																									  26\!: &\bn   771\\
																																									  27\!: &\bn  7878\\
																																									  28\!: &\bn 19241\\
																																									  29\!: &\bn  7984\\
																																								\tcw{2}30\!: &\bn   216} 
																																												& \tf{26\!: &\bn    96\\
																																													  27\!: &\bn  2035\\
																																													  28\!: &\bn 13440\\
																																													  29\!: &\bn 20057\\
																																												\tcw{2}30\!: &\bn  3808\\
																																															&		}   
\end{tabular}
}
$ $\\[5mm]
\resizebox{\textwidth}{!}{
\begin{tabular}{r|r|r|r|r|r|r|r|r|r|r|r}
$f_0$  				& 17       		& 18        	& 19      	 	& 20         	& 21			& 22			& 23			& 24		   & 25  		  	& 26			& 27		 \\ \hline
\# of $(f_0,*,*,9)$ 	& 6				& 5			   	& 5 			& 4				& 4				& 3				& 3				& 2			   & 2			  	& 1				& 1			 \\
\# of combin. types	& 38007			& 32492			& 24741			& 16747			& 10069			& 5306			& 2468			& 946		   & 331		  	& 76			& 23		 \\[10mm] \hline
\tf{f_2\!:&\bn\text{\# combin. types}\\
		 & 					  \\
		 &					  \\
		 &					  \\
		 &					  \\
		 & }	& \tf{26\!:&\bn    7\\
		 			  27\!:&\bn  268\\
					  28\!:&\bn 4047\\
					  29\!:&\bn18090\\
					  30\!:&\bn14763\\
				\tcw{2}31\!:&\bn  832}
					  			& \tf{27\!:&\bn   23\\
		 			 		   		  28\!:&\bn  596\\
					 				  29\!:&\bn 6519\\
									  30\!:&\bn18482\\
								\tcw{2}31\!:&\bn 6872\\
									  	   &	}
					  				  			& \tf{28\!:&\bn   45\\
		 			 		   						  29\!:&\bn 1057\\
					 								  30\!:&\bn 8578\\
													  31\!:&\bn13559\\
												\tcw{2}32\!:&\bn 1502\\
									  					   &	}
					  				  					   		&  \tf{29\!:&\bn   84\\
		 			 		   		 						   		   30\!:&\bn 1574\\
					 				 								   31\!:&\bn 8793\\
									 								   32\!:&\bn 6296\\
									 						    \tcw{241\!:}&\bn\tcw{1}\\
									 								   		&	}
																				&  \tf{30\!:&\bn  128\\
		 			 		   		 						   		   		 		   31\!:&\bn 2016\\
					 				 								   				   32\!:&\bn 6536\\
									 								   				   33\!:&\bn 1389\\
									 										    \tcw{241\!:}&\bn\tcw{1}\\
									 								   						&	}
																								&   \tf{31\!:&\bn  172\\
		 			 		   		 						   											32\!:&\bn 2064\\
					 				 																	33\!:&\bn 3070\\
									 																		 &		  \\
									 														     \tcw{241\!:}&\bn\tcw{1}\\
									 								   						   				 &	}
																											    &	\tf{32\!:&\bn  212\\
																	 		   		 						   		    33\!:&\bn 1563\\
																				 				 					    34\!:&\bn  693\\
																								 							 &		  \\
																								 			     \tcw{241\!:}&\bn\tcw{1}\\
																								 			   				 &	}
																															 	&  \tf{33\!:&\bn  209\\
		 			 		   		 						   																 		   34\!:&\bn  737\\
					 				 																								   		&		 \\
									 																								   	 	&		 \\
									 																						    \tcw{241\!:}&\bn\tcw{1}\\
									 								   														  				&	}
					  				  																											&  \tf{34\!:&\bn  163\\
		 			 		   		 						   																						   35\!:&\bn  168\\
					 				 																												   		&		 \\
									 																												   	 	&		 \\
									 																										    \tcw{241\!:}&\bn\tcw{1}\\
									 								   																		  				&	}
																																								&  \tf{35\!:&\bn   76\\
		 			 		   		 						   																										   		&		 \\
					 				 																																   	 	&		 \\
									 																																   	 	&		 \\
									 						  																								    \tcw{241\!:}&\bn\tcw{1}\\
									 								   																				   					 	&	}
																																									 			&  \tf{36\!:&\bn   23\\
		 			 		   		 						   																														   		&		 \\
					 				 																																				   	 	&		 \\
									 																																				   	 	&		 \\
									 						  																												    \tcw{241\!:}&\bn\tcw{1}\\
									 								   																										  				&	}
\end{tabular}
}
\vspace{0.5cm}
\caption{Complete enumeration of $f$-vectors of $4$-polytopes with $f_3=9$.}
\label{tab:fvector_f0_is_9}

\end{table}

In each of \Cref{tab:fvector_f0_is_7} and \Cref{tab:fvector_f0_is_8}, an entry is surrounded by ``$\star$'' symbols, since they are of particular interest:

\begin{itemize}[leftmargin= 70pt]
	\item[(14,28,21,7)] 
		The dual of the cyclic polytope $C_4(7)$ has this $f$-vector, we discuss two strategies to find an inscribed realization in \Cref{ssec:47interpol} and \Cref{ssec:47stereo}. 	
		It is well-known that the $f$-vector of $C_4(7)$ and its dual determine combinatorially unique polytopes	\cite[Chapter~6.3]{gruenbaum_convex_1967} and \cite[Table~6]{firsching_complete_2018}). 
	\item[(20,40,28,8)] 
		Besides the cyclic polytope $C_4(8)$, there are precisely two other neighborly $4$-polytopes on $8$ vertices \cite{altshuler_complete_1985}. 
		Their associated dual polytopes are the only polytopes that have this $f$-vector. 
		We show in \Cref{ssec:48obstruction} that the dual of $C_4(8)$ is not inscribable and that the other two duals are not inscribable in \Cref{sec:neighborly}. 
		In particular, the $f$-vector $(20,40,28,8)$ is not inscribable.
\end{itemize}
	
\begin{remark}
\label{rem:small_fvector}
We verified that the $20$ combinatorially distinct $4$-polytopes with $f_0+f_3\leq 15$ are all inscribable, so the associated $13$ $f$-vectors are also inscribable.
In combination with \Cref{cor:main}, the smallest $f$-vector (in terms of the sum $f_0+f_3$) that is not inscribable, must satisfy $16 \leq f_0+f_3 \leq 28$.
According to the complete classifications given in \cite[Table~2.3]{brinkmann_fvector_2016} and \cite[Table~6 and~7]{firsching_complete_2018}, there are ten $f$-vectors that satisfy $16 \leq f_0+f_3 \leq 19$ and that determine a combinatorially unique polytope. 
These $f$-vectors are:
	
	\begin{itemize}[leftmargin= 20pt]
	    \item $(9,26,26,9)$\\
	        We provide an inscribed realization of this polytope below.
		\item $(7,18,19,8)$, $(7,17,19,9)$, $(9,19,17,7)$ and $(7,18,22,11)$\\
			Inscribing coordinates for these $f$-vectors are provided in \Cref{appendix:inscribed_coordinates}.
		\item $( 8,19,20, 9)$ and $( 9,20,19, 8)$\\
			If we label the vertices of the first polytope by $1,\ldots,8$, then the facets are five tetrahedra, $1234$, $2568$, $2578$, $2678$ and $5678$,
			and four $3$-faces $12356$, $12457$, $134567$ and $23467$.
			Labeling the vertices of the second polytope $1,\ldots,9$, the facets are four tetrahedra, $1234$, $1235$, $1345$ and $6789$, and
			four $3$-faces $1245689$, $234678$, $235679$ and $345789$.
		\item  $(9,20,20, 9)$\\
			If we label the vertices by $1,\ldots,9$, then the facets of this self-dual polytope are the five tetrahedra, $1234$, $5678$, $5689$, $5789$, $6789$,
			and four $3$-faces $123567$, $124579$, $134679$ and $234569$. 
		\item $(11,22,18,7)$ and $(12,25,20,7)$\\
			Inscribing these polytopes involves many degrees of freedom (lots of vertices) and many constraints (many vertices per facet). 
			The difficulty of this task is less than, but comparable to, the quest of inscribing $C_4(7)$.
	\end{itemize}
We invite the reader to find inscribed realizations for the polytopes with $f$-vector $( 8,19,20, 9)$, $( 9,20,19, 8)$ and $( 9,20,20, 9)$ using a combination of the elementary polytope constructions pyramid, bipyramid, truncation and their dual operations.
There are more $f$-vectors that determine a combinatorially unique polytope, but for these $f_0+f_3\geq 20$, putting them outside the range of fully classified combinatorial types.
The difficulty in finding inscribed realizations varies significantly. 
The polytope with \(f\)-vector \((13,28,22,7)\) is very hard to inscribe, 
but the \(f\)-vectors $(10,25,28,13)$ and $(13,28,25,10)$ are easy enough to inscribe, realizations are provided below.
\end{remark}

In the remainder of this section we present inscribed realizations of three $4$-polytopes that are uniquely determined by the $f$-vectors: $T_1$, determined by $(9,26,26,9)$, $T_2$, determined by $(10,25,28,13)$, and its dual, $T_2^*$, determined by $(13,28,25,10)$.
This shows these \(f\)-vectors are inscribable.

\noindent
\textbf{Inscribed realization of $T_1$ with $f_{T_1} = ( 9,26,26, 9)$.}\\
% sage: T3 = Polyhedron(vertices = [[0,0,0,-1],[1,0,0,0],[-1,0,0,0],[0,1,0,0],[0,-1,0,0],[0,0,1,0],[0,0,-1,0],[1/2,1/2,1/2,-1/2],[-1/2,-1/2,-1/2,-1/2]])
% sage: T3.f_vector()
% (1, 9, 26, 26, 9, 1)
The $f$-vector $f_{T_1}$ has a unique associated combinatorial type of $4$-polytope, see~\cite[Table~7]{firsching_complete_2018}.
If we label the vertices $0,1,\ldots,8$, then the facets are two tetrahedra $0123$ and $5678$ and $3$-faces
\[
	01245,\quad 01346,\quad 02347,\quad 123567,\quad 24578,\quad 14568\quad\text{and}\quad 34678.
\]
To realize this polytope (as a Schlegel projection into the facet $123567$), start with an octahedron $123567$ (with diagonals $17$, $26$ and $35$), 
and cone over vertex~$4$, placed in its center. This decomposes the octahedron into eight tetrahedra. 
Now stellarly subdivide tetrahedron $1234$ (resp.\ $4567$) into four tetrahedra by placing a vertex $0$ (resp.\ $8$) in its center. 
Now vertex~$4$ is contained in twelve tetrahedra, the remaining two tetrahedra, $0123$ and $5678$, are facets. 
Moving $0$ and $8$ sufficiently close to the centers of triangles $123$ and $567$ we can ensure that these twelve tetrahedra group together in pairs along triangles $124$, $134$, $234$, $456$, $457$ and $467$ to form six triangular bipyramids.
This $f$-vector is inscribable as an inscribed realization for $T_1$ is given by the following coordinates:
{\small
\[
\begin{blockarray}{rrrrrrrrr}
\begin{block}{rrrrrrrrr}
0 & 1 & 2 & 3 & 4 & 5 & 6 & 7 & 8 \\
\end{block}
\begin{block}{(rrrrrrrrr)}
-\frac{1}{2} & -1 & 0 & 0 & 0 & 0 & 0 & 1 & \frac{1}{2}\topstrut  \\[0.25em]
-\frac{1}{2} & 0 & -1 & 0 & 0 & 0 & 1 & 0 & \frac{1}{2}  \\[0.25em]
-\frac{1}{2} & 0 & 0 & -1 & 0 & 1 & 0 & 0 & \frac{1}{2}  \\[0.25em]
-\frac{1}{2} & 0 & 0 & 0 & -1 & 0 & 0 & 0 & -\frac{1}{2}\botstrut \\
\end{block}
\end{blockarray}\ .
\]
}

\medskip
\noindent
\textbf{Inscribed realization of $T_2$ with $f_{T_2} = (10,25,28,13)$.}\\
% sage: T4 = Polyhedron(vertices = [[1,1,1,1],[-1,-1,1,1],[1,-1,-1,1],[-1,1,-1,1],[1,1,1,-1],[-1,-1,1,-1],[1,-1,-1,-1],[-1,1,-1,-1],[0,0,10/13,-24/13],[0,0,-10/13,-24/13]])
% sage: T4.f_vector()
% (1, 10, 25, 28, 13, 1)
The $f$-vector $f_{T_2}$ has a unique associated combinatorial type of $4$-polytope~$T_2$, see~\cite[Table~2.3]{brinkmann_fvector_2016}. 
If we label the vertices $0,1,\ldots,9$, then the facets are nine tetrahedra 
\[
	0123,\quad 4568,\quad 4579,\quad 4589,\quad 4679,\quad 4689,\quad 5678,\quad 5789\quad\text{and}\quad 6789
\]
and $3$-faces
\[
  012456,\quad 013457,\quad 023467\quad\text{and}\quad 123567.
\]
To realize this polytope, start with the boundary complex of the $4$-dimensional prism over base tetrahedra $0123$ and $4567$ 
(with facets $0123$, $012456$, $013457$, $023467$, $123567$ and $4567$). Then subdivide tetrahedron $4567$ by placing two vertices, 
$8$ and $9$, onto the line segment connecting the two mid-points of $47$ and $56$, and cone to the other four edges to obtain tetrahedra 
$4589$, $4689$, $5789$, and $6789$. Complete the subdivision of $4567$ by adding tetrahedra $4568$ and $5678$, and $4579$ and $4679$.
This $f$-vector is also inscribable as an inscribed realization for $T_2$ is given by:
{\small
\[
\begin{blockarray}{rrrrrrrrrr}
\begin{block}{rrrrrrrrrr}
0 & 1 & 2 & 3 & 4 & 5 & 6 & 7 & 8 & 9\\
\end{block}
\begin{block}{(rrrrrrrrrr)}
 -1 & -1 &  1 &  1 & -1 & -1 &  1 & 1 & 0 & 0\topstrut \\
 -1 &  1 & -1 &  1 & -1 &  1 & -1 & 1 & 0 & 0 \\
  1 & -1 & -1 &  1 &  1 & -1 & -1 & 1 & -\frac{10}{13} & \frac{10}{13} \\[0.25em]
 -1 & -1 & -1 & -1 &  1 &  1 &  1 & 1 & -\frac{24}{13} & -\frac{24}{13} \botstrut \\
\end{block}
\end{blockarray}\ .
\]
}

\medskip
\noindent
\textbf{Inscribed realization of $T_2^*$ with $f_{T_2^*} = (13,28,25,10)$.}\\
% sage: T4star = Polyhedron(vertices = [[2,2,2,1],[-2,-2,2,1],[2,-2,-2,1],[-2,2,-2,1],[2,2,2,-1],[-2,-2,2,-1],[2,-2,-2,-1],[-2,2,-2,-1],[0,2,0,3],[0,-2,0,3],[0,0,2,3],[0,0,-2,3],[6/37,6/37,6/37,-133/37]])
% sage: T4star.f_vector()
% (1, 13, 28, 25, 10, 1)
We assume that the vertices of $T_2^*$ are labeled $0,1,\dots, 8,9,\text{A,B,C}$. 
Then the facets for $T_2^*$ are four tetrahedra
$028\text{C}$, $02\text{AC}$, $08\text{AC}$ and $28\text{AC}$
and six $3$-faces
\[
	01234589,\quad 012367\text{AB},\quad 014689\text{AB},\quad 134567,\quad 235789\text{AB}\quad\text{and}\quad 45679\text{B}.
\]
An inscribed realization for $T_2^*$ is given by the following coordinates:
{\small
\[
\begin{blockarray}{rrrrrrrrrrrrr}
\begin{block}{rrrrrrrrrrrrr}
0 & 1 & 2 & 3 & 4 & 5 & 6 & 7 & 8 & 9 & \text{A} & \text{B} & \text{C}\\
\end{block}
\begin{block}{(rrrrrrrrrrrrr)}
-2 & -2 & -2 & -2 & 0 & 0 & 0 & 0 & 2 & 2 & 2 & 2 & \frac{6}{37}\topstrut \\[0.25em]
-2 & -2 & 2 & 2 & -2 & 0 & 0 & 2 & -2 & -2 & 2 & 2 & \frac{6}{37} \\[0.25em]
2 & 2 & -2 & -2 & 0 & -2 & 2 & 0 & -2 & -2 & 2 & 2 & \frac{6}{37} \\[0.25em]
-1 & 1 & -1 & 1 & 3 & 3 & 3 & 3 & -1 & 1 & -1 & 1 & -\frac{133}{37}\botstrut \\
\end{block}
\end{blockarray}\ .
\]
}
Since $T_2$ and $T_2^*$ are dual to each other, this $f$-vector and its dual, $(10,25,28,13)$, are thus inscribable \emph{and} circumscribable.

%%%%%%%%%%%%%%%%%%%%%
%%%%%%%%%%%%%%%%%%%%%
\section{Circumscribability of cyclic polytopes}
\label{sec:cyclic_poly}
%%%%%%%%%%%%%%%%%%%%%
%%%%%%%%%%%%%%%%%%%%%

In this section, we study circumscribability of cyclic polytopes or, equivalently, inscribability of their duals.
The $d$-dimensional cyclic polytope on $k$ vertices is denoted by $C_d(k)$. 
Its combinatorial type is realized by the convex hull of $k$ increasing distinct points on the moment curve $\nu_d: \mathbb{R} \to \mathbb{R}^d$ sending $t$ to $(t,t^2, \ldots , t^d)$. Its facets are described purely combinatorially using Gale's evenness condition, see e.g.\ \cite[Chapter~0]{ziegler_lectures_1995}. 
The dual of the cyclic polytope~$C_d(k)$ is denoted by $C_d(k)^*$. 

We fix a labeling of the faces of $C_d(k)$ and of its dual:
We order the $k$ vertices of $C_d(k)$ and identify them with the numbers $\{1,2,\dots, k\}$.
The facets of its dual $C_d(k)^*$ are identified with an additional star $i^*$.
Each face of $C_d(k)$ is labeled by the set of vertex labels it contains.
To write a face label $\{i,j,k,l, \dots\}$ of $C_d(k)$, we abuse notation and write $ijkl\cdots$. 
For the corresponding dual face $\{i,j,k,l,\dots\}^*$ of $C_d(k)^*$, we write $(ijkl\cdots)^*$.
As a consequence, by taking facet intersections in the dual, faces of $C_d(k)^*$ are labeled by the set of facets they are contained in. In particular, the vertices of $C_d(k)^*$ are labeled by subsets of $\{1,2,\dots,k\}$ corresponding to facets of $C_d(k)$.

In dimension $d=4$, facets of $C_4(k)^*$ are combinatorially equivalent to $C_3(k-1)^*$, a wedge over a $(k-2)$-gon.
Motivated by Lemma~\ref{lem:classic_i}, we first look at the inscribed realization space these wedges in \Cref{ssec:wedges}.
In \Cref{ssec:47interpol,ssec:47stereo} we present two proofs that the cyclic polytope $C_4(7)$ is circumscribable.
In \Cref{ssec:48obstruction}, we show that the cyclic polytope $C_4(8)$ is not circumscribable using a geometric obstruction.
Finally, in \Cref{ssec:larger_obstruction} we use the argument for $C_4(8)$ and Gale's evenness condition to extend this obstruction to cyclic polytopes $C_d(k)$, where $k \geq d+4 \geq 8$. 

%%%%%%%%%%%%%%%%%%%%%
\subsection{The inscribed realization space of wedges over polygons}
\label{ssec:wedges}
%%%%%%%%%%%%%%%%%%%%%

In this section we describe the space of inscribed realizations of the facets of $C_4(k)^*$.
They are combinatorially isomorphic to a wedge over a $(k-2)$-gon, denoted by \(F_k\).

Inscribed realizations of $3$-polytopes up to M\"obius transformations correspond to feasible solutions of a set of linear constraints imposed on the set of external dihedral angles at the edges of the polytope \cite{rivin_characterization_1996}. 
As a corollary of Rivin's work, the realization space of a $3$-polytope up to M\"obius transformations is contractible. 
This does not extend to higher dimensions where universality holds \cite{adiprasito_universality_2015}.

The wedge $F_k$ has $f$-vector $(2k-6,3k-9,k-1)$.
Its facets consist of two $(k-2)$-gons, two triangles and $k-5$ quadrilaterals. 
Following \cite{richtergebert_realization_1996}, the dimension of the realization space of a $3$-polytope up to M\"obius transformation is $f_{1}-6$. 
Hence, the realization space of $F_k$ has dimension $3k-15$. 
The \emph{inscribed} realization space of $F_k$ up to M\"obius transformations has dimension $k-3$. For reasonably small $k$ this can be checked computationally using Rivin's linear program. 
It follows that the inscribed realization space of $F_k$ up to Euclidean isometries and homotheties is of dimension $k$.

The construction of explicit coordinates for an inscribed realization of $C_4(7)^*$ (see \Cref{ssec:47stereo}) is based on the following parametrizations of the space of inscribed realizations of~$F_k$, up to M\"obius transformations and up to Euclidean isometries and homotheties. 

\begin{proposition}
  \label{prop:moebius}
  The inscribed realization space of the wedge over a $(k-2)$-gon $F_k$ up to M\"obius transformations is homeomorphic to
  $$ \inter (\Delta^{k-5}) \times (0,\pi) \times I,$$
  where $\inter (\Delta^{k-5})$ denotes the interior of a $(k-5)$-dimensional simplex, $(0,\pi)$ determines the angle between the two $(k-2)$-gons of $F_k$, and $I$ is an open interval only depending on the position of the vertices of one of the $(k-2)$-gon of $F_k$.

  In particular, the realization space of $F_k$ is homeomorphic to an open $(k-3)$-ball.
\end{proposition}

\begin{proof}[Sketch of proof]
We refer the reader to the picture on the left in \Cref{fig:proj}.
Assume that $F_k$ is inscribed on $\mathbb{S}^2$. 
We use stereographic projection $\pi_N$ as described in \Cref{ssec:background}. 
After applying a suitable M\"obius transformation we can assume that the two vertices of $F_k$ contained in the two $(k-2)$-gons are mapped to north pole $N$ and south pole~$S$ of~$\mathbb{S}^2$ and that a third point determining the circle $c \subset \mathbb{S}^2$ of the first $(k-2)$-gon of $F_k$ is mapped to $(1,0)$. 
We are now free to arbitrarily choose $k-5$ points on $c$ between points $(1,0)$ and $N$. 
That is, we choose $k-5$ points on a line segment yielding the first (and largest) factor of the inscribed realization space $\inter(\Delta^{k-5})$.

An inscribed realization of $F_k$ is now determined by the position of one more vertex, $q_0 \in \mathbb{S}^2 \setminus c$ (after M\"obius transformations in the ``front hemisphere'' of $\mathbb{S}^2$): The triple $(q_0,S,N)$ determines a circle $d \subset S$ containing all the vertices of the second $(k-2)$-gon. Moreover, the triple $(q_0,(1,0),p_1)$ determines a circle $e \subset S$ containing all vertices of one of the quadrilaterals of $F_k$. It now follows that the fourth point of this quadrilateral, $q_1$, is determined as the intersection $d \cap e$. By iteration, the coordinates of the remaining $k-5$ vertices of $F_k$ are determined and lead to at most one inscribed realization. 

If $q_0$ is chosen at the latitude of $(1,0)$, this configuration leads, in fact, to an inscribed realization for all longitudes strictly between $0$ and $\pi$: symmetry around the $SN$-axis of $\mathbb{S}^2$ shows that all $q_i$ must then have the same latitude as all $p_i$, $1 \leq i \leq k-5$. The same is true for a starting latitude of $q_0$ contained in a sufficiently small interval around the latitude of $(1,0)$. Denote such a latitude as \emph{valid}. In general, for a given latitude to be valid, $q_{i-1}$ must be further ``south'' than $q_{i}$, $1 \leq i \leq k-5$. 
It follows that the set of valid latitudes is an open interval, as moving $q_0$ ``north'' (resp.\ ``south'') eventually causes $q_1$ to move past $q_0$ (resp.\ $q_2$), and likewise for further \(q_i\). 

Moreover, a valid latitude for $q_0$ is not affected by varying the longitude of $q_0$ (i.e., by varying the opening angle of the wedge~$F_k$): For instance, observe that the line segment $q_0 q_1$ under rotation around the $SN$-axis must remain on both the planes defined by $(q_0,(1,0),p_1)$ and $(N,S,q_0)$, and $q_1$ must remain on $\mathbb{S}^2$ with fixed latitude. In particular, longitude and latitude of $q_0$ can be described by points in $(0,\pi) \times I$.

Altogether, every point in $ \inter (\Delta^{k-5}) \times (0,\pi) \times I$ corresponds to a unique inscribed realization of $F_k$. 
Conversely, since $N$, $S$, $(1,0)$ and the hemisphere of $q_0$ are fixed, an inscribed realization of $F_k$ up to M\"obius transformations corresponds to a unique point in $ \inter (\Delta^{k-5}) \times (0,\pi) \times I$.
\end{proof}

\begin{figure}[!htbp]
  \begin{center}
    \raisebox{1.88cm}{\includegraphics[height=6cm]{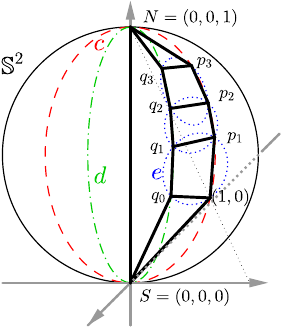}} \hfill \includegraphics[height=8cm]{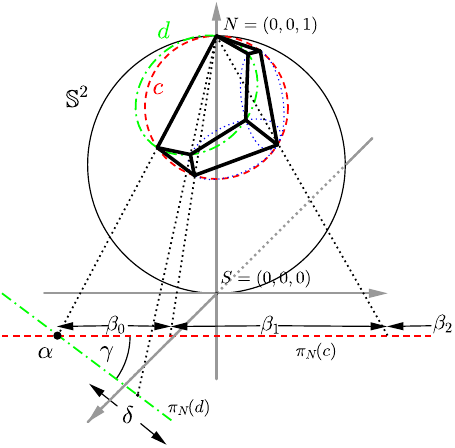}
  \end{center}
\caption{Left: Inscribed realization space of $F_k$, $k=8$, up to M\"obius transforms. Right: Inscribed realization of $F_k$, $k=7$, up to Euclidean isometries and homotheties. Parameters as given in  \Cref{prop:wedge_real} are indicated.\label{fig:proj}}
\end{figure}

The next result provides a parametrization of the inscribed realizations of $F_k$ up to Euclidean isometries and homotheties, which implies that it is contractible.
We use this parametrization in \Cref{ssec:47stereo}.

\begin{proposition}
\label{prop:wedge_real}
The inscribed realization space of the wedge over a $(k-2)$-gon $F_k$ up to Euclidean isometries and homotheties is parametrized by
\[
\mathcal{R}_{F_k}:=\{(\alpha,\beta,\gamma,\delta)~:~\alpha\in\R^2,\ \beta\in (\R^+)^{k-4},\ \gamma\in(0,\pi), \delta \in I_{\alpha,\beta,\gamma}\} ,
\]
where $I_{\alpha,\beta,\gamma}$ is an open interval.
In particular, $\mathcal{R}_{F_k}$ is contractible and has dimension $k$.
\end{proposition}

\begin{proof}[Sketch of proof]

  We use the same setup as in the previous statement. 
  The main differences are that we now have to account for three more degrees of freedom.

We denote the circle that contains the vertices of one \((k-2)\)-gon \(c\) and the other such circle \(d\).  
After applying a suitable transformation we can assume that one vertex of the wedge edge is at $N$ and $\pi_N (c)$ is a line parallel to the $x$-axis. 
We can further assume that the projection of the other vertex of the wedge edge has 
the smallest \(x\)-value among the vertices on \(\pi_N (c)\) and 
the smallest \(y\)-value among the vertices on \(\pi_N (d)\). 
In the previous proof we always had $\alpha = (0,0)$ but here it can be freely chosen, adding two of the extra three degrees of freedom. 
The third extra degree of freedom arises from now placing the remaining $(k-4)$ vertices 
(instead of $(k-5)$ vertices in the case of M\"obius transformations) of the first $(k-2)$-gon onto $\pi_N(c)$ on the positive \(x\) side of $\alpha$. 
This yields an open $(k-2)$-ball $(\alpha,\beta)$.

The remaining two parameters relate to the second \((k-2)\)-gon. 
The first of these parameters is the angle between the lines \(\pi_N (c)\) and \(\pi_N (d)\), 
denoted by \(\gamma\), which can take any value between \(0\) and \(\pi\). 
Finally, we let \(I_{\alpha,\beta,\gamma}\) be the set of positions for the third vertex of the second $(k-2)$-gon that determine a valid inscribed realization.

It is apparent that this parameterization has dimension \(k\).
 The statement now follows from observing that for a fixed \(\gamma\), 
 \(I_{\alpha,\beta,\gamma}\) is determined by strict linear inequalities in \(\alpha\) and \(\beta\), and contains at least one element. 
 In particular, it contains the point \(v=((0,0),\beta^c,\beta^c,\gamma)\) for any fixed \(\beta^c\). 
 Therefore \(I_{\alpha,\beta,\gamma}\) is an open polyhedron for each choice of \(\gamma\), and is contractable to \(v\). 
 The set of \(v\) for all values of \(\gamma\) is an open segment, which is contractible, so the realization space \(\mathcal{R}_{F_k}\) is contractible.
 \end{proof}

%%%%%%%%%%%%%%%%%%%%%
\subsection{Circumscribing $C_4(7)$ using interpolation}
\label{ssec:47interpol}
%%%%%%%%%%%%%%%%%%%%%

One approach to test whether a given polytope is circumscribable consists in writing facet normals in terms of the vertex coordinates and checking if they lie on a sphere. 
This is the same as checking that the dual of a circumscribed polytope is inscribed.

This approach works well for the cyclic polytope $C_4(7)$ because the facet normals (at least generically) uniquely determine a quadratic hypersurface by interpolation. 
We consider real quadratic forms in $(n+1)$ variables and the action of $\gl_{n+1}(\R)$ on the vector space of all quadratic forms given by change of coordinates, i.~e.~$(M.q)(x) = q(Mx)$. 
If a quadratic form~$q$ is represented by the symmetric matrix $A$, that is $q(x) = x^T A x$, then $M\in\gl_{n+1}(\R)$ acts on~$A$ via $M. A := M^T A M$.

\begin{proposition}
    Let $q(x) = x^T A x$ be a quadratic form in $(n+1)$ variables $x_0,x_1,\ldots,x_n$.
    The quadratic form $q$ can be transformed into the quadratic form defined by $x_0^2 - \sum_{i=1}^n x_i^2$ over~$\R$ if and only if the signature of $A$ is $(1,n)$, i.~e.~$A$ has $1$ positive and $n$ negative eigenvalues.
\end{proposition}

\begin{proof}
    This is Sylvester's law of inertia, see \cite[Section~20.3]{DymLA}.
\end{proof}

Let $t_1 < t_2 < \dots < t_7$ be the values defining the vertices of $C_4(7)$ on the moment curve, and recall that $i$ denotes the vertex $(t_i,t_i^2,t_i^3,t_i^4)$, where $1\leq i \leq 7$.
The $14$ facets of $C_4(7)$ can be obtained by Gale's evenness condition:
\[
    1234, 1237, 1245, 1256, 1267, 1347, 1457, 1567, 2345, 2356, 2367, 3456, 3467, 4567.
\]
The facet normal vectors of the facets $ijkl$ can be computed by Cramer's rule as the kernel of the matrix
\[
    \begin{pmatrix}
        1 & t_i & t_i^2 & t_i^3 & t_i^4 \\
        1 & t_j & t_j^2 & t_j^3 & t_j^4 \\
        1 & t_k & t_k^2 & t_k^3 & t_k^4 \\
        1 & t_l & t_l^2 & t_l^3 & t_l^4 
    \end{pmatrix}.
\]
This gives $14$ points, \(\{r_i\}_{i=1, \ldots, 14}\), in $\R\P^4$ that we want to place on a quadratic hypersurface. 
Since the vector space of quadratic forms in $5$ variables has dimension $15$, $14$ generic points uniquely determine a quadratic form vanishing at these $14$ points and we can compute its equation using Lagrange interpolation. To set this up, let $\texttt{m}$ be the row vector of the $15$ monomials of degree~$2$ in $5$ variables in a fixed order. Writing $r_i$ for the $14$ points in $\R\P^4$, we create the $14\times 15$ matrix $(\texttt{m}(r_i))_{i=1,\ldots,14}$. The coefficient vectors of the quadratic forms vanishing at these $14$ points are exactly the elements of the kernel of this matrix.

\begin{proposition}\label{prop:interpolation}
    Let $t_1 = 0$, $t_2 = 1$, $t_3 = 3$, $t_4 = 7$, $t_5 = 11$, $t_6 = 13$, $t_7 = 21$. 
    The representing matrix of the (up to scaling unique) quadratic form vanishing on the $14$ facet normals of the cyclic polytope $C_4(7)$ 
    defined by these $7$ values is 
    \[M=
      \begin{pmatrix}
        22237 & 130328 & 1323281 & 15129020 & 184061477\\
        130328 & 339339 & 2534532 & 27498471 & 344552208\\
        1323281 & 2534532 & 12297285 & 106450344 & 1304584281\\
        15129020 & 27498471 & 106450344 & 677359683 & 7142515380\\
        184061477 & 344552208 & 1304584281 & 7142515380 & 59989246317
      \end{pmatrix}. 
    \]
    The quadratic form associated to this matrix has signature $(1,4)$.
\end{proposition}

\begin{proof}
    This result can be computed in the way described above. 
    The fact that the quadratic form vanishing at these $14$ points is unique up to scaling is equivalent to the fact that the matrix 
    $(\texttt{m}(r_i))_{i\in \{1,\ldots,14\}}$ has rank $14$. 
\end{proof}

Since the cyclic polytope $C_4(7)$ is the only combinatorial type with $f$-vector $(7,21,28,14)$ 
(see e.g.\ \cite[Chapter~6.3]{gruenbaum_convex_1967} and \cite[Table~6]{firsching_complete_2018}), 
the following result says that every polytope with this $f$-vector is strongly circumscribable in the sense of Chen and Padrol \cite[Section~2.1]{chen_scribability_2017}.
\begin{theorem}
    The $4$-dimensional cyclic polytope with $7$ vertices is strongly circumscribable.
\end{theorem}

\begin{proof}
  For the choice of parameters $t_1 = 0$, $t_2 = 1$, $t_3 = 3$, $t_4 = 7$, $t_5 = 11$, $t_6 = 13$, $t_7 = 21$, 
  the corresponding cyclic polytope $C_4(7) = \conv\{\nu_4(t_i)\colon i=1,\ldots,7\}$ 
  has the property that the outer facet normal vectors embedded into $\R\P^4$ via $x\mapsto (1:x)$ lie on a 
  quadric defined by the quadratic form represented by $M$ of signature $(1,4)$, by \Cref{prop:interpolation}. 
  This means that the facets of this realization of $C_4(7)$ are tangent to the quadric hypersurface projectively dual to the quadric defined by $M$, 
  which is given by the inverse of $M$, see for example \cite[Chapter~1]{GKZ}. 
  The signature of the inverse matrix is still $(1,4)$, implying that the given realization is circumscribed to a quadric 
  with the signature of the quadratic form $x_0^2 - x_1^2 - x_2^2 - x_3^2 - x_4^2$. 
  This quadric can therefore be transformed by a projective transformation into the unit sphere in $\R^4$ embedded in $\R\P^4$ via $x\mapsto (1:x)$.
  This shows the weak circumscribability of \(C_4(7)\).
  
  We will show strong circumscribability by transforming \(M\) to its rational canonical form.
    The following matrix 
    \[Q=
      \begin{pmatrix}
        1 & -\frac{130328}{22237} & \frac{10792760433}{858136931} & -\frac{5642531895}{81241733} & \frac{505870365}{2359789} \\
        0 & 1 & -\frac{10554561644}{858136931} & \frac{6857048059}{81241733} & -\frac{1313262412}{2359789} \\
        0 & 0 & 1 & -\frac{1509354537}{81241733} & \frac{506139238}{2359789} \\
        0 & 0 & 0 & 1 & -\frac{62410180}{2359789} \\
        0 & 0 & 0 & 0 & 1
      \end{pmatrix}
    \]
    makes \(Q^\top MQ\) diagonal, the rational canonical form of \(M\).
  The quadratic form determined by \(Q^\top MQ\) is nondegenerate, and therefore removing the origin from the quadric yields two connected components.
  To prove strong circumscribability, we need to show that the vertices of \(C_4(7)^*\) still lie on a common component after the transformation.
     Since the only positive entry on the diagonal of \(Q^\top MQ\) is the first one, and the first entry of \(Q^{-1}r_i\) is negative for every \(i \in \{1,\ldots,14\}\), 
     all the \(r_i\) lie on the negative connected component of the quadric determined by \(Q^\top MQ\).
\end{proof}

%%%%%%%%%%%%%%%%%%%%%
\subsection{Circumscribing $C_4(7)$ using stereographic projection}
\label{ssec:47stereo}
%%%%%%%%%%%%%%%%%%%%%

In this section, we present a circumscribed realization of \(C_4(7)\) with explicit coordinates for its stereographic projection through a well-chosen vertex. 
To do this we rely on the realization space of the inscribed wedge described in \Cref{ssec:wedges}.

Consider the polytope \(C_4(7)^*\). 
Up to cyclic symmetry of $[7]$, there are two combinatorial types of vertices in \(C_4(7)^*\).
The first type consists of the seven vertices in the orbit of \((1234)^*\), and the second, of those in the orbit of \((1245)^*\). 
We stereographically project \(C_4(7)^*\) from the vertex $(1234)^*$ onto a generic hyperplane. 
By \Cref{lem:projection}, since $C_4(7)^*$ is simple, the image of the three facets labeled $5^*,6^*$, and $7^*$ form a polytopal subdivision of a convex tetrahedron.
Further, if $C_4(7)^*$ is inscribed, then the resulting subdivision is Delaunay~\cite[Proposition~13]{gonska_inscribable_2013}.
The result of the projection is combinatorially equivalent to the subdivision illustrated in \Cref{fig:c47_facets}.

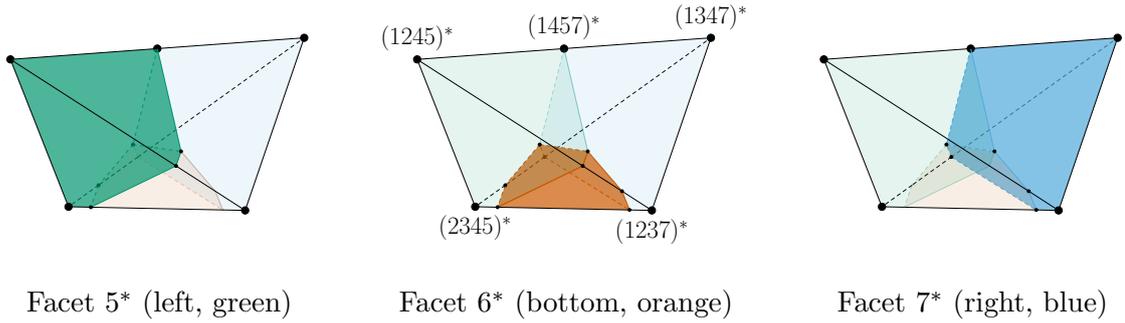
\begin{figure}[!htbp]
\begin{tabular}{ccc}
\resizebox{5cm}{!}{\begin{tikzpicture}%
	[x={(-0.134723cm, -0.671867cm)},
	y={(0.926403cm, -0.346219cm)},
	z={(0.351607cm, 0.654772cm)},
	scale=4.000000,
	back/.style={dashed, thin},
	edges/.style={color=black, thick, line cap=round},
	edge1/.style={color=red!95!black, thick, line cap=round, opacity=0.300000},
	facet1/.style={fill=red!95!black,fill opacity=0.100000},
	edge2/.style={color=skyblue!95!black, thick, line cap=round, opacity=0.300000},
	facet2/.style={fill=skyblue!95!black,fill opacity=0.100000},
	edge3/.style={color=bluishgreen!95!black, thick},
	facet3/.style={fill=bluishgreen!95!black,fill opacity=0.700000},
	vertex/.style={inner sep=0.075cm,circle,draw=black,fill=black,thick}]
%
%
%% Coordinate of the vertices:
%%
\coordinate (-0.18005, -0.38308, 0.23385) at (-0.18005, -0.38308, 0.23385);
\coordinate (-0.43288, -0.25022, 0.49956) at (-0.43288, -0.25022, 0.49956);
\coordinate (-0.45500, -0.16954, 0.38070) at (-0.45500, -0.16954, 0.38070);
\coordinate (-0.66358, -0.38357, 1.26578) at (-0.66358, -0.38357, 1.26578);
\coordinate (-1.61456, 0.73100, 1.00000) at (-1.61456, 0.73100, 1.00000);
\coordinate (0.00000, 0.00000, 1.00000) at (0.00000, 0.00000, 1.00000);
\coordinate (0.02909, -0.54549, 0.12215) at (0.02909, -0.54549, 0.12215);
\coordinate (0.03594, -0.39935, 0.20129) at (0.03594, -0.39935, 0.20129);
\coordinate (0.07645, 0.46432, 0.66900) at (0.07645, 0.46432, 0.66900);
\coordinate (0.08331, 0.61046, 0.74814) at (0.08331, 0.61046, 0.74814);
\coordinate (0.10927, 0.34217, 0.84782) at (0.10927, 0.34217, 0.84782);
\coordinate (0.14342, -0.01056, 0.97887) at (0.14342, -0.01056, 0.97887);
\coordinate (0.28740, -1.49815, 1.53156) at (0.28740, -1.49815, 1.53156);

%% Drawing edges in the back
%%
\draw[edges,back] (-0.45500, -0.16954, 0.38070) -- (-1.61456, 0.73100, 1.00000);
%% Drawing the facets
%%
\fill[facet2] (0.08331, 0.61046, 0.74814) -- (-1.61456, 0.73100, 1.00000) -- (-0.66358, -0.38357, 1.26578) -- (0.00000, 0.00000, 1.00000) -- (0.10927, 0.34217, 0.84782) -- cycle {};
\fill[facet2] (0.00000, 0.00000, 1.00000) -- (-0.66358, -0.38357, 1.26578) -- (-0.43288, -0.25022, 0.49956) -- cycle {};
\fill[facet2] (0.07645, 0.46432, 0.66900) -- (-0.45500, -0.16954, 0.38070) -- (-0.43288, -0.25022, 0.49956) -- (0.00000, 0.00000, 1.00000) -- (0.10927, 0.34217, 0.84782) -- cycle {};
\fill[facet2] (0.08331, 0.61046, 0.74814) -- (0.10927, 0.34217, 0.84782) -- (0.07645, 0.46432, 0.66900) -- cycle {};
\node[vertex,label=above:{\color{white}\huge $(1347)^*$}] at (-1.61456, 0.73100, 1.00000) {};
\node[vertex,label=below:{\color{white}\huge $(1237)^*$}] at (0.08331, 0.61046, 0.74814) {};

%% Drawing edges in the front of Polytope 2
%%
\draw[edges] (-0.66358, -0.38357, 1.26578) -- (-1.61456, 0.73100, 1.00000);
\draw[edge2] (-0.66358, -0.38357, 1.26578) -- (0.00000, 0.00000, 1.00000);
\draw[edges] (-1.61456, 0.73100, 1.00000) -- (0.08331, 0.61046, 0.74814);
\draw[edge2] (0.00000, 0.00000, 1.00000) -- (0.10927, 0.34217, 0.84782);
\draw[edge2] (0.10927, 0.34217, 0.84782) -- (0.07645, 0.46432, 0.66900);
\draw[edges] (0.10927, 0.34217, 0.84782) -- (0.08331, 0.61046, 0.74814);
\draw[edges] (0.07645, 0.46432, 0.66900) -- (0.08331, 0.61046, 0.74814);
%% Drawing edges in the back
%%
\draw[edge1,back] (-0.45500, -0.16954, 0.38070) -- (-0.43288, -0.25022, 0.49956);
\draw[edges,back] (-0.45500, -0.16954, 0.38070) -- (-0.18005, -0.38308, 0.23385);
\draw[edge1,back] (-0.45500, -0.16954, 0.38070) -- (0.07645, 0.46432, 0.66900);
%% Drawing the facets
%%
\fill[facet1] (0.14342, -0.01056, 0.97887) -- (0.00000, 0.00000, 1.00000) -- (-0.43288, -0.25022, 0.49956) -- (-0.18005, -0.38308, 0.23385) -- (0.03594, -0.39935, 0.20129) -- cycle {};
\fill[facet1] (0.10927, 0.34217, 0.84782) -- (0.14342, -0.01056, 0.97887) -- (0.03594, -0.39935, 0.20129) -- (0.07645, 0.46432, 0.66900) -- cycle {};
\fill[facet1] (0.14342, -0.01056, 0.97887) -- (0.00000, 0.00000, 1.00000) -- (0.10927, 0.34217, 0.84782) -- cycle {};
%% Drawing edges in the front
%%
\draw[edge1] (0.00000, 0.00000, 1.00000) -- (0.10927, 0.34217, 0.84782);
\draw[edge1] (0.00000, 0.00000, 1.00000) -- (0.14342, -0.01056, 0.97887);
\draw[edges] (0.03594, -0.39935, 0.20129) -- (0.07645, 0.46432, 0.66900);
\draw[edge1] (0.10927, 0.34217, 0.84782) -- (0.07645, 0.46432, 0.66900);
\draw[edges] (0.10927, 0.34217, 0.84782) -- (0.14342, -0.01056, 0.97887);
%% Drawing edges in the back
%%
\draw[edge3,back] (-0.66358, -0.38357, 1.26578) -- (-0.43288, -0.25022, 0.49956);
\draw[edge3,back] (-0.43288, -0.25022, 0.49956) -- (-0.18005, -0.38308, 0.23385);
\draw[edge3,back] (-0.43288, -0.25022, 0.49956) -- (0.00000, 0.00000, 1.00000);
\draw[edge3,back] (-0.18005, -0.38308, 0.23385) -- (0.03594, -0.39935, 0.20129);
\draw[edges,back] (-0.18005, -0.38308, 0.23385) -- (0.02909, -0.54549, 0.12215);
\node[vertex,inner sep=0.03cm] at (-0.18005, -0.38308, 0.23385) {};
\node[vertex,inner sep=0.03cm] at (-0.43288, -0.25022, 0.49956) {};
%% Drawing the facets
%%
\fill[facet3] (0.28740, -1.49815, 1.53156) -- (0.02909, -0.54549, 0.12215) -- (0.03594, -0.39935, 0.20129) -- (0.14342, -0.01056, 0.97887) -- cycle {};
\node[vertex,label=above:{\color{white}\huge $(1457)^*$}] at (-0.66358, -0.38357, 1.26578) {};
\node[vertex,label=below:{\color{white}\huge $(2345)^*$}] at (0.02909, -0.54549, 0.12215) {};
\node[vertex,label=above:{\color{white}\huge $(1245)^*$}] at (0.28740, -1.49815, 1.53156) {};
\fill[facet3] (0.28740, -1.49815, 1.53156) -- (-0.66358, -0.38357, 1.26578) -- (0.00000, 0.00000, 1.00000) -- (0.14342, -0.01056, 0.97887) -- cycle {};
%% Drawing edges in the front
%%
\draw[edge3] (-0.66358, -0.38357, 1.26578) -- (0.00000, 0.00000, 1.00000);
\draw[edges] (-0.66358, -0.38357, 1.26578) -- (0.28740, -1.49815, 1.53156);
\draw[edge3] (0.00000, 0.00000, 1.00000) -- (0.14342, -0.01056, 0.97887);
\draw[edges] (0.03594, -0.39935, 0.20129) -- (0.02909, -0.54549, 0.12215);
\draw[edge3] (0.03594, -0.39935, 0.20129) -- (0.14342, -0.01056, 0.97887);
\draw[edges] (0.02909, -0.54549, 0.12215) -- (0.28740, -1.49815, 1.53156);
\draw[edges] (0.14342, -0.01056, 0.97887) -- (0.28740, -1.49815, 1.53156);
\node[vertex,inner sep=0.03cm] at (0.03594, -0.39935, 0.20129) {};
\node[vertex,inner sep=0.03cm] at (0.14342, -0.01056, 0.97887) {};
\node[vertex,inner sep=0.03cm] at (0.00000, 0.00000, 1.00000) {};
\end{tikzpicture}} & \resizebox{5cm}{!}{\begin{tikzpicture}%
	[x={(-0.134723cm, -0.671867cm)},
	y={(0.926403cm, -0.346219cm)},
	z={(0.351607cm, 0.654772cm)},
	scale=4.000000,
	back/.style={dashed, thin},
	edges/.style={color=black, thick, line cap=round},
	edge1/.style={color=red!95!black, thick, line cap=round},
	facet1/.style={fill=red!95!black,fill opacity=0.700000},
	edge2/.style={color=skyblue!95!black, thick, line cap=round, opacity=0.300000},
	facet2/.style={fill=skyblue!95!black,fill opacity=0.100000},
	edge3/.style={color=bluishgreen!95!black, thick, line cap=round, opacity=0.300000},
	facet3/.style={fill=bluishgreen!95!black,fill opacity=0.100000},
	vertex/.style={inner sep=0.075cm,circle,draw=black,fill=black,thick}]
%
%
%% Coordinate of the vertices:
%%
\coordinate (-0.18005, -0.38308, 0.23385) at (-0.18005, -0.38308, 0.23385);
\coordinate (-0.43288, -0.25022, 0.49956) at (-0.43288, -0.25022, 0.49956);
\coordinate (-0.45500, -0.16954, 0.38070) at (-0.45500, -0.16954, 0.38070);
\coordinate (-0.66358, -0.38357, 1.26578) at (-0.66358, -0.38357, 1.26578);
\coordinate (-1.61456, 0.73100, 1.00000) at (-1.61456, 0.73100, 1.00000);
\coordinate (0.00000, 0.00000, 1.00000) at (0.00000, 0.00000, 1.00000);
\coordinate (0.02909, -0.54549, 0.12215) at (0.02909, -0.54549, 0.12215);
\coordinate (0.03594, -0.39935, 0.20129) at (0.03594, -0.39935, 0.20129);
\coordinate (0.07645, 0.46432, 0.66900) at (0.07645, 0.46432, 0.66900);
\coordinate (0.08331, 0.61046, 0.74814) at (0.08331, 0.61046, 0.74814);
\coordinate (0.10927, 0.34217, 0.84782) at (0.10927, 0.34217, 0.84782);
\coordinate (0.14342, -0.01056, 0.97887) at (0.14342, -0.01056, 0.97887);
\coordinate (0.28740, -1.49815, 1.53156) at (0.28740, -1.49815, 1.53156);
%% Drawing edges in the back
%%
\draw[edges,back] (-0.45500, -0.16954, 0.38070) -- (-1.61456, 0.73100, 1.00000);
%% Drawing the facets
%%
\fill[facet2] (0.08331, 0.61046, 0.74814) -- (-1.61456, 0.73100, 1.00000) -- (-0.66358, -0.38357, 1.26578) -- (0.00000, 0.00000, 1.00000) -- (0.10927, 0.34217, 0.84782) -- cycle {};
\fill[facet2] (0.00000, 0.00000, 1.00000) -- (-0.66358, -0.38357, 1.26578) -- (-0.43288, -0.25022, 0.49956) -- cycle {};
\fill[facet2] (0.07645, 0.46432, 0.66900) -- (-0.45500, -0.16954, 0.38070) -- (-0.43288, -0.25022, 0.49956) -- (0.00000, 0.00000, 1.00000) -- (0.10927, 0.34217, 0.84782) -- cycle {};
\fill[facet2] (0.08331, 0.61046, 0.74814) -- (0.10927, 0.34217, 0.84782) -- (0.07645, 0.46432, 0.66900) -- cycle {};
\node[vertex,label=above:{\huge $(1347)^*$}] at (-1.61456, 0.73100, 1.00000) {};
\node[vertex,label=below:{\huge $(1237)^*$}] at (0.08331, 0.61046, 0.74814) {};
%% Drawing edges in the front of Polytope 2
%%
\draw[edges] (-0.66358, -0.38357, 1.26578) -- (-1.61456, 0.73100, 1.00000);
\draw[edge2] (-0.66358, -0.38357, 1.26578) -- (0.00000, 0.00000, 1.00000);
\draw[edges] (-1.61456, 0.73100, 1.00000) -- (0.08331, 0.61046, 0.74814);
\draw[edge2] (0.00000, 0.00000, 1.00000) -- (0.10927, 0.34217, 0.84782);
\draw[edge2] (0.10927, 0.34217, 0.84782) -- (0.07645, 0.46432, 0.66900);
\draw[edges] (0.10927, 0.34217, 0.84782) -- (0.08331, 0.61046, 0.74814);
\draw[edges] (0.07645, 0.46432, 0.66900) -- (0.08331, 0.61046, 0.74814);
%% Drawing edges in the back
%%
\draw[edge1,back] (-0.45500, -0.16954, 0.38070) -- (-0.43288, -0.25022, 0.49956);
\draw[edges,back] (-0.45500, -0.16954, 0.38070) -- (-0.18005, -0.38308, 0.23385);
\draw[edge1,back] (-0.45500, -0.16954, 0.38070) -- (0.07645, 0.46432, 0.66900);
\node[vertex,inner sep=0.03cm] at (-0.45500, -0.16954, 0.38070) {};
%% Drawing the facets
%%
\fill[facet1] (0.14342, -0.01056, 0.97887) -- (0.00000, 0.00000, 1.00000) -- (-0.43288, -0.25022, 0.49956) -- (-0.18005, -0.38308, 0.23385) -- (0.03594, -0.39935, 0.20129) -- cycle {};
\fill[facet1] (0.10927, 0.34217, 0.84782) -- (0.14342, -0.01056, 0.97887) -- (0.03594, -0.39935, 0.20129) -- (0.07645, 0.46432, 0.66900) -- cycle {};
\fill[facet1] (0.14342, -0.01056, 0.97887) -- (0.00000, 0.00000, 1.00000) -- (0.10927, 0.34217, 0.84782) -- cycle {};
%% Drawing edges in the front
%%
\draw[edge1] (0.00000, 0.00000, 1.00000) -- (0.10927, 0.34217, 0.84782);
\draw[edge1] (0.00000, 0.00000, 1.00000) -- (0.14342, -0.01056, 0.97887);
\draw[edges] (0.03594, -0.39935, 0.20129) -- (0.07645, 0.46432, 0.66900);
\draw[edge1] (0.03594, -0.39935, 0.20129) -- (0.14342, -0.01056, 0.97887);
\draw[edge1] (0.10927, 0.34217, 0.84782) -- (0.07645, 0.46432, 0.66900);
\draw[edges] (0.10927, 0.34217, 0.84782) -- (0.14342, -0.01056, 0.97887);
%% Drawing edges in the back
%%
\draw[edge3,back] (-0.66358, -0.38357, 1.26578) -- (-0.43288, -0.25022, 0.49956);
\draw[edge1,back] (-0.43288, -0.25022, 0.49956) -- (-0.18005, -0.38308, 0.23385);
\draw[edge1,back] (-0.43288, -0.25022, 0.49956) -- (0.00000, 0.00000, 1.00000);
\draw[edge1,back] (-0.18005, -0.38308, 0.23385) -- (0.03594, -0.39935, 0.20129);
\draw[edges,back] (-0.18005, -0.38308, 0.23385) -- (0.02909, -0.54549, 0.12215);
%% Drawing the facets
%%
\fill[facet3] (0.28740, -1.49815, 1.53156) -- (0.02909, -0.54549, 0.12215) -- (0.03594, -0.39935, 0.20129) -- (0.14342, -0.01056, 0.97887) -- cycle {};
\node[vertex,label=above:{\huge $(1457)^*$}] at (-0.66358, -0.38357, 1.26578) {};
\node[vertex,label=below:{\huge $(2345)^*$}] at (0.02909, -0.54549, 0.12215) {};
\node[vertex,label=above:{\huge $(1245)^*$}] at (0.28740, -1.49815, 1.53156) {};
\fill[facet3] (0.28740, -1.49815, 1.53156) -- (-0.66358, -0.38357, 1.26578) -- (0.00000, 0.00000, 1.00000) -- (0.14342, -0.01056, 0.97887) -- cycle {};
%% Drawing edges in the front
%%
\draw[edge3] (-0.66358, -0.38357, 1.26578) -- (0.00000, 0.00000, 1.00000);
\draw[edges] (-0.66358, -0.38357, 1.26578) -- (0.28740, -1.49815, 1.53156);
\draw[edge1] (0.00000, 0.00000, 1.00000) -- (0.14342, -0.01056, 0.97887);
\draw[edges] (0.03594, -0.39935, 0.20129) -- (0.02909, -0.54549, 0.12215);
\draw[edge1] (0.03594, -0.39935, 0.20129) -- (0.14342, -0.01056, 0.97887);
\draw[edges] (0.02909, -0.54549, 0.12215) -- (0.28740, -1.49815, 1.53156);
\draw[edges] (0.14342, -0.01056, 0.97887) -- (0.28740, -1.49815, 1.53156);
\node[vertex,inner sep=0.03cm] at (-0.18005, -0.38308, 0.23385) {};
\node[vertex,inner sep=0.03cm] at (-0.43288, -0.25022, 0.49956) {};
\node[vertex,inner sep=0.03cm] at (0.03594, -0.39935, 0.20129) {};
\node[vertex,inner sep=0.03cm] at (0.14342, -0.01056, 0.97887) {};
\node[vertex,inner sep=0.03cm] at (0.00000, 0.00000, 1.00000) {};
\node[vertex,inner sep=0.03cm] at (0.10927, 0.34217, 0.84782) {};
\node[vertex,inner sep=0.03cm] at (0.07645, 0.46432, 0.66900) {};

\end{tikzpicture}} & \resizebox{5cm}{!}{\begin{tikzpicture}%
	[x={(-0.134723cm, -0.671867cm)},
	y={(0.926403cm, -0.346219cm)},
	z={(0.351607cm, 0.654772cm)},
	scale=4.000000,
	back/.style={dashed, thin},
	edges/.style={color=black, thick, line cap=round},
	edge1/.style={color=red!95!black, thick, line cap=round, opacity=0.300000},
	facet1/.style={fill=red!95!black,fill opacity=0.100000},
	edge2/.style={color=skyblue!95!black, thick, line cap=round},
	facet2/.style={fill=skyblue!95!black,fill opacity=0.700000},
	edge3/.style={color=bluishgreen!95!black, thick, line cap=round, opacity=0.300000},
	facet3/.style={fill=bluishgreen!95!black,fill opacity=0.100000},
	vertex/.style={inner sep=0.075cm,circle,draw=black,fill=black,thick}]
%
%
%% Coordinate of the vertices:
%%
\coordinate (-0.18005, -0.38308, 0.23385) at (-0.18005, -0.38308, 0.23385);
\coordinate (-0.43288, -0.25022, 0.49956) at (-0.43288, -0.25022, 0.49956);
\coordinate (-0.45500, -0.16954, 0.38070) at (-0.45500, -0.16954, 0.38070);
\coordinate (-0.66358, -0.38357, 1.26578) at (-0.66358, -0.38357, 1.26578);
\coordinate (-1.61456, 0.73100, 1.00000) at (-1.61456, 0.73100, 1.00000);
\coordinate (0.00000, 0.00000, 1.00000) at (0.00000, 0.00000, 1.00000);
\coordinate (0.02909, -0.54549, 0.12215) at (0.02909, -0.54549, 0.12215);
\coordinate (0.03594, -0.39935, 0.20129) at (0.03594, -0.39935, 0.20129);
\coordinate (0.07645, 0.46432, 0.66900) at (0.07645, 0.46432, 0.66900);
\coordinate (0.08331, 0.61046, 0.74814) at (0.08331, 0.61046, 0.74814);
\coordinate (0.10927, 0.34217, 0.84782) at (0.10927, 0.34217, 0.84782);
\coordinate (0.14342, -0.01056, 0.97887) at (0.14342, -0.01056, 0.97887);
\coordinate (0.28740, -1.49815, 1.53156) at (0.28740, -1.49815, 1.53156);
%% Drawing edges in the back
%%
\draw[edges,back] (-0.45500, -0.16954, 0.38070) -- (-1.61456, 0.73100, 1.00000);
%% Drawing the facets
%%
\fill[facet2] (0.08331, 0.61046, 0.74814) -- (-1.61456, 0.73100, 1.00000) -- (-0.66358, -0.38357, 1.26578) -- (0.00000, 0.00000, 1.00000) -- (0.10927, 0.34217, 0.84782) -- cycle {};
\fill[facet2] (0.00000, 0.00000, 1.00000) -- (-0.66358, -0.38357, 1.26578) -- (-0.43288, -0.25022, 0.49956) -- cycle {};
\fill[facet2] (0.07645, 0.46432, 0.66900) -- (-0.45500, -0.16954, 0.38070) -- (-0.43288, -0.25022, 0.49956) -- (0.00000, 0.00000, 1.00000) -- (0.10927, 0.34217, 0.84782) -- cycle {};
\fill[facet2] (0.08331, 0.61046, 0.74814) -- (0.10927, 0.34217, 0.84782) -- (0.07645, 0.46432, 0.66900) -- cycle {};
\node[vertex,label=above:{\color{white}\huge $(1347)^*$}] at (-1.61456, 0.73100, 1.00000) {};
\node[vertex,label=below:{\color{white}\huge $(1237)^*$}] at (0.08331, 0.61046, 0.74814) {};
%% Drawing edges in the front of Polytope 2
%%
\draw[edges] (-0.66358, -0.38357, 1.26578) -- (-1.61456, 0.73100, 1.00000);
\draw[edges] (-1.61456, 0.73100, 1.00000) -- (0.08331, 0.61046, 0.74814);
\draw[edges] (0.10927, 0.34217, 0.84782) -- (0.08331, 0.61046, 0.74814);
\draw[edges] (0.07645, 0.46432, 0.66900) -- (0.08331, 0.61046, 0.74814);
%% Drawing edges in the back
%%
\draw[edge2,back] (-0.45500, -0.16954, 0.38070) -- (-0.43288, -0.25022, 0.49956);
\draw[edges,back] (-0.45500, -0.16954, 0.38070) -- (-0.18005, -0.38308, 0.23385);
\draw[edge2,back] (-0.45500, -0.16954, 0.38070) -- (0.07645, 0.46432, 0.66900);
%% Drawing the facets
%%
\fill[facet1] (0.14342, -0.01056, 0.97887) -- (0.00000, 0.00000, 1.00000) -- (-0.43288, -0.25022, 0.49956) -- (-0.18005, -0.38308, 0.23385) -- (0.03594, -0.39935, 0.20129) -- cycle {};
\fill[facet1] (0.10927, 0.34217, 0.84782) -- (0.14342, -0.01056, 0.97887) -- (0.03594, -0.39935, 0.20129) -- (0.07645, 0.46432, 0.66900) -- cycle {};
\fill[facet1] (0.14342, -0.01056, 0.97887) -- (0.00000, 0.00000, 1.00000) -- (0.10927, 0.34217, 0.84782) -- cycle {};
%% Drawing edges in the front
%%
\draw[edge2] (0.00000, 0.00000, 1.00000) -- (0.10927, 0.34217, 0.84782);
\draw[edges] (0.03594, -0.39935, 0.20129) -- (0.07645, 0.46432, 0.66900);
\draw[edge2] (0.10927, 0.34217, 0.84782) -- (0.07645, 0.46432, 0.66900);
\draw[edges] (0.10927, 0.34217, 0.84782) -- (0.14342, -0.01056, 0.97887);
%% Drawing edges in the back
%%
\draw[edge2,back] (-0.66358, -0.38357, 1.26578) -- (-0.43288, -0.25022, 0.49956);
\draw[edge3,back] (-0.43288, -0.25022, 0.49956) -- (-0.18005, -0.38308, 0.23385);
\draw[edge2,back] (-0.43288, -0.25022, 0.49956) -- (0.00000, 0.00000, 1.00000);
\draw[edge3,back] (-0.18005, -0.38308, 0.23385) -- (0.03594, -0.39935, 0.20129);
\draw[edges,back] (-0.18005, -0.38308, 0.23385) -- (0.02909, -0.54549, 0.12215);
%% Drawing the facets
%%
\fill[facet3] (0.28740, -1.49815, 1.53156) -- (0.02909, -0.54549, 0.12215) -- (0.03594, -0.39935, 0.20129) -- (0.14342, -0.01056, 0.97887) -- cycle {};
\node[vertex,label=above:{\color{white}\huge $(1457)^*$}] at (-0.66358, -0.38357, 1.26578) {};
\node[vertex,label=below:{\color{white}\huge $(2345)^*$}] at (0.02909, -0.54549, 0.12215) {};
\node[vertex,label=above:{\color{white}\huge $(1245)^*$}] at (0.28740, -1.49815, 1.53156) {};
\fill[facet3] (0.28740, -1.49815, 1.53156) -- (-0.66358, -0.38357, 1.26578) -- (0.00000, 0.00000, 1.00000) -- (0.14342, -0.01056, 0.97887) -- cycle {};
%% Drawing edges in the front
%%
\draw[edge2] (-0.66358, -0.38357, 1.26578) -- (0.00000, 0.00000, 1.00000);
\draw[edges] (-0.66358, -0.38357, 1.26578) -- (0.28740, -1.49815, 1.53156);
\draw[edge3] (0.00000, 0.00000, 1.00000) -- (0.14342, -0.01056, 0.97887);
\draw[edges] (0.03594, -0.39935, 0.20129) -- (0.02909, -0.54549, 0.12215);
\draw[edge3] (0.03594, -0.39935, 0.20129) -- (0.14342, -0.01056, 0.97887);
\draw[edges] (0.02909, -0.54549, 0.12215) -- (0.28740, -1.49815, 1.53156);
\draw[edges] (0.14342, -0.01056, 0.97887) -- (0.28740, -1.49815, 1.53156);
\node[vertex,inner sep=0.03cm] at (-0.43288, -0.25022, 0.49956) {};
\node[vertex,inner sep=0.03cm] at (-0.45500, -0.16954, 0.38070) {};
\node[vertex,inner sep=0.03cm] at (0.00000, 0.00000, 1.00000) {};
\node[vertex,inner sep=0.03cm] at (0.10927, 0.34217, 0.84782) {};
\node[vertex,inner sep=0.03cm] at (0.07645, 0.46432, 0.66900) {};
\end{tikzpicture}} \\[1em]
Facet $5^*$ (left, green) & Facet $6^*$ (bottom, orange) & Facet $7^*$ (right, blue)
\end{tabular}
\caption{The three facets of $C_4(7)^*$ that do not contain the vertex $(1234)^*$.}
\label{fig:c47_facets}
\end{figure}

We focus on facet $6^*$, emphasized in the middle of \Cref{fig:c47_facets}.
We assume \(C_4(7)^*\) to be inscribed, and use the parametrization of \Cref{ssec:wedges} to realize facet $6^*$ in $\R^3$ using $7$ variables.
Observe that the location of the four vertices of the tetrahedron are determined by facet equations of the realization of facet $6^*$. 
This way, twelve of the thirteen vertices contained in the tetrahedron are determined.
The remaining vertex $(1457)^*$, located on the top edge of the tetrahedron, still has one degree of freedom.
Lemma~\ref{lem:classic_iii} together with Lemma~\ref{lem:projection_iv} imply that every pentagonal face is inscribed.
The vertex $(2345)^*$ and the pentagon $(56)^*$ determine a unique $2$-sphere containing those six points. 
Vertices $(1245)^*$ and $(1457)^*$ must be on this $2$-sphere giving two equations of degree $(2,2,2,2,1,2,2)$ and $(2,2,3,2,1,0,0)$.
Similarly, the vertex $(1237)^*$ and the pentagon $(67)^*$ determine a unique $2$-sphere containing vertices $(1457)^*$ and $(1347)^*$ leading to equations of degree $(13,13,26,8,4,20,20)$ and $(6,6,11,3,2,10,10)$.
This leads to an underdetermined system of $4$ equations in $7$ variables.

To reduce the complexity, we impose symmetry, resulting in a system with fewer variables.
Indeed, this reduces the parameter space to just 4 variables: $\alpha$ and the first two lengths $\beta_1,\beta_2$.
To eliminate the angle $\gamma$, we require that the projection of the great circle passing through $(0,0)$ and $\alpha$ be the angle bisector of the two rays (see \Cref{fig:proj} on the right for an illustration of the parameter space $(\alpha,\beta,\gamma)$ -- for a different choice of $\gamma$).
Since one of the rays is horizontal, knowledge of $\alpha$ prescribes the angle $\gamma$.
A further constraint comes from the fact that facets~$5^*$ and $7^*$ must be isometric and hence vertex $(1457)^*$ must be in the middle of the edge $\pi_{(1234)^*}((14)^*)$.

Taking the educated guess $\alpha=(-3/2,-1/2)$, we compute the intersection of the two constraints, to obtain two algebraic curves on the plane with degrees $(5,3)$ and $(9,7)$. 
Newton's method and subsequent verification then results in exact coordinates for the stereographic projection of an inscribed embedding of \(C_4(7)^*\) from vertex $(1234)^*$.
The coordinates are given in \Cref{app:c47}.
The realization has coordinates in \(\mathbb{Q}[a]\), where \(a\) is the solution to a degree 10 polynomial.
Therefore, the corresponding inscribed realization of $C_4(7)^*$ in $\R^4$ must have degree at least~$20$.

\begin{question}
Is there a rational inscribed realization of $C_4(7)^*$? If not, what is the smallest possible degree of the coordinates as algebraic numbers over $\Q$?
\end{question}

In particular, a degree 2 realization is of exceptional interest.

%%%%%%%%%%%%%%%%%%%%%
\subsection{Non-circumscribability of $C_4(8)$}
\label{ssec:48obstruction}
%%%%%%%%%%%%%%%%%%%%%

We start by giving a classical result related to inscribability. 
It is due to Jakob Steiner, originally proved by Auguste Miquel, see \Cref{fig:miquel} for an illustration.

\begin{lemma}[{Miquel's theorem \cite[Theorem 1.6 and Theorem 18.5]{richtergebert_perspectives_2011}}]
Let $p_i$, $1 \leq i \leq 8$, be eight distinct points in $\R^2$ such that the following quadruples are cocircular: \((p_1,p_2,p_3,p_4),\) \((p_1,p_2,p_5,p_6),\) \((p_2,p_3,p_6,p_7),\) \((p_3,p_4,p_7,p_8),\) \((p_1,p_4,p_5,p_8)\).
Then \((p_5,p_6,p_7,p_8)\) is cocircular.
\end{lemma}

\begin{figure}[!htbp]
\begin{tikzpicture}[%
	vertex/.style={ivnner sep=0pt,circle,draw=black!25!black,fill=black!75!black,thick}]

\def\innerdist{0.25}
\def\outerdist{1.1}

\coordinate (0) at (0,0);

\coordinate (a) at (\innerdist,\innerdist);
\coordinate (b) at (\innerdist,-\innerdist);
\coordinate (c) at (-\innerdist,-\innerdist);
\coordinate (d) at (-\innerdist,\innerdist);

\coordinate (A) at (\outerdist,\outerdist);
\coordinate (B) at (\outerdist,-\outerdist);
\coordinate (C) at (-\outerdist,-\outerdist);
\coordinate (D) at (-\outerdist,\outerdist);

\node (circle1) [name path=circle1,draw,circle through=(a)] at (A) {};
\node (circle2) [name path=circle2,draw,circle through=(b)] at (B) {};
\node (circle3) [name path=circle3,draw,circle through=(c)] at (C) {};
\node (circle4) [name path=circle4,draw,circle through=(d)] at (D) {};

\path [name intersections={of=circle1 and circle2, by={1, 5}}];
\path [name intersections={of=circle2 and circle3, by={2, 6}}];
\path [name intersections={of=circle3 and circle4, by={3, 7}}];
\path [name intersections={of=circle4 and circle1, by={8, 4}}];

\node[inner sep=0pt,label=below right:$p_1$] at (1) {};
\node[inner sep=1.5mm,label=above:$p_2$] at (2) {};
\node[inner sep=0pt,label=below left:$p_3$] at (3) {};
\node[inner sep=1.5mm,label=below:$p_4$] at (4) {};
\node[inner sep=1mm,label=right:$p_5$] at (5) {};
\node[inner sep=1.5mm,label=below:$p_6$] at (6) {};
\node[inner sep=0pt,label=left:$p_7$] at (7) {};
\node[inner sep=1.5mm,label=above:$p_8$] at (8) {};

\node (circle0) [name path=circle0,draw,circle through=(1)] at (0) {};

\foreach \point in {1,2,3,4,5,6,7,8}
  \fill [red] (\point) circle (2pt);

\node (circle5) [name path=circle5,draw,dashed,red, thick,circle through=(5)] at (0) {};

\end{tikzpicture}
\caption{The dashed circle is the sixth circle passing through four points}
\label{fig:miquel}
\end{figure}
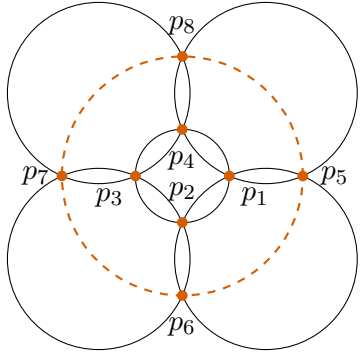

Miquel's theorem lifts to a statement about planarity of points on a 2-sphere.

\begin{lemma}[Miquel's theorem, spherical version]
\label{lem:Miquel}
Let \(p_i\), $1 \leq i \leq 8$, be eight distinct points on $\mathbb{S}^2$ such that the following quadruples of vertices are coplanar: \((p_1,p_2,p_3,p_4),\) \((p_1,p_2,p_5,p_6),\) \((p_2,p_3,p_6,p_7),\) \((p_3,p_4,p_7,p_8),\) \((p_1,p_4,p_5,p_8)\). 
Then the quadruple \((p_5,p_6,p_7,p_8)\) is coplanar and thus the lines spanned by \((p_5,p_6)\) and \((p_7,p_8)\) are coplanar.
\end{lemma}

Miquel's theorem describes the underlying reason for the fact that one cannot force a facet of $C_4(8)^*$ to be inscribed.

\begin{theorem}
\label{thm:c48}
No realization of $C_4(8)^*$ has an inscribed facet, although all its facets are inscribable.
\end{theorem}

\begin{proof}
The facets of $C_4(8)^*$ are all combinatorially equivalent to $F_8$, a wedge over a hexagon.
By \Cref{prop:wedge_real}, they are inscribable.
By Gale's evenness condition, the facets of $C_4(8)$ are given by
\begin{align*}
   & 1234, 1238, 1245, 1256, 1267, 1278, 1348, 1458, 1568, 1678, \\
   & 2345, 2356, 2367, 2378, 3456, 3467, 3478, 4567, 4578, 5678.
\end{align*}
By duality, these correspond to the vertices of $C_4(8)^*$, and we write $(ijkl)^*$ for the vertex of~$C_4(8)^*$ corresponding to facet $ijkl$ of $C_4(8)$.
By \Cref{lem:projection}, projecting $C_4(8)^*$ stereographically from vertex $(3467)^*$ yields a polytopal subdivision of a tetrahedron $\Delta$ into four copies of $F_8$, see \Cref{fig:c48}.

\begin{figure}[!htbp]
\input{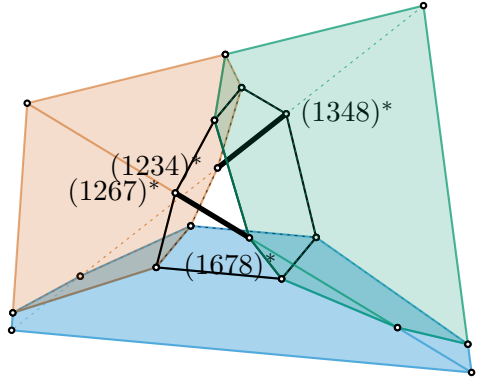}
\caption{The image of the stereographic projection of $C_4(8)^*$ from vertex $(3467)^*$. Facets $1^*$, $2^*$, $5^*$ and $8^*$ are drawn in white (center), orange (top left), blue (bottom), and green (top right) respectively.}
\label{fig:c48}
\end{figure}

By definition, edge $(134)^*$ belongs to wedges $1^*$, $3^*$ and $4^*$, while edge $(167)^*$ belongs to wedges $1^*$, $6^*$, and $7^*$.
Since $\Delta$ is convex, these two edges $(134)^*$ and $(167)^*$ must be skew, as can be seen in the stereographic projection, see \Cref{fig:c48}.
If the wedge $1^*$ is inscribed, the squares $(12)^*$, $(14)^*$, $(15)^*$, $(16)^*$, and $(18)^*$ are inscribed on a common $2$-sphere. 
By \Cref{lem:Miquel}, the four vertices $(1234)^*$, $(1267)^*$, $(1348)^*$, and $(1678)^*$ are then coplanar, forcing $(134)^*$ and $(167)^*$ to be both coplanar and skew which is impossible.

Since the dimension is even, facets of $C_4(8)^*$ are related through combinatorial automorphisms of the cyclic polytope. 
Hence, for each facet there is an appropriate choice of vertex that provides the required configuration in the stereographical projection. 
\end{proof}

\begin{corollary}
\label{cor:c48}
The cyclic polytope $C_4(8)$ is not circumscribable.
\end{corollary}

\begin{proof}
Since $C_4(8)^*$ has no realization with an inscribed facet, $C_4(8)^*$ is not inscribable and~$C_4(8)$ is not circumscribable by Lemma~\ref{lem:classic}.
\end{proof}

%%%%%%%%%%%%%%%%%%%%%
\subsection{Larger cyclic polytopes $C_d(k)$}
\label{ssec:larger_obstruction}
%%%%%%%%%%%%%%%%%%%%%

The obstruction in the case of \(C_4(8)\) appears as a subcomplex in a large class of cyclic polytopes.
On the one hand, by taking specific successive stereographic projections until the resulting object is a tetrahedron, the lines spanned by opposite edges of the tetrahedron are skew.
On the other hand, the tetrahedron contains a projected face whose inscription forces these skew lines to be coplanar, leading to a contradiction.

\begin{theorem}
\label{thm:no_real_c48}
Let $k\geq 8$. No realization of $C_4(k)^*$ has an inscribed facet, although all its facets are inscribable. 
\end{theorem}

\begin{proof}
The proof follows the proof of \Cref{thm:c48}.
Set \(P = C_4(k)^*\) and denote by $s$ the vertex $(3467)^*$ of $P$.
Using \Cref{lem:projection}, \(\pi_s(P)\) defines a subdivision of a tetrahedron with triangles \(\pi_s(3^*),\pi_s(4^*),\pi_s(6^*),\pi_s(7^*)\).
Notice that 
\begin{center}
the image \(\pi_s((34)^*)\) and the image \(\pi_s((67)^*)\) are skew. \quad $(\star)$  
\end{center}

Consider the polytope \(\pi_s(1^*)\). 
Since facet \(1^*\) does not contain vertex~\(s\), by \Cref{lem:projection} the polytope \(1^*\)~and \(\pi_s(1^*)\) are combinatorially isomorphic.
The facets of \(1^*\) are ridges of \(P\) and are labeled as \((1i)^*\) for some $1 \leq i \leq k$.
Among the facets of \(\pi_s(1^*)\) are $\pi_s((14)^*)$, $\pi_s((16)^*)$, $\pi_s((15)^*)$, $\pi_s((12)^*)$, and $\pi_s((1k)^*)$.
See \Cref{fig:subcube} for an illustration.

\begin{figure}[!htbp]
\input{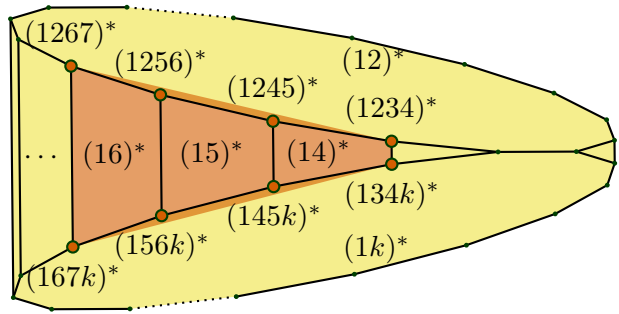}
\caption{A wedge and a subcomplex formed by eight vertices.}
\label{fig:subcube}
\end{figure}

Assume, for the sake of contradiction, that the facet $1^*$ is inscribed.
By \Cref{lem:projection}, its projection \(\pi_s(1^*)\) is also inscribed.
By Lemma~\ref{lem:classic_iii}, the five polygons $\pi_s((14)^*)$, $\pi_s((16)^*)$, $\pi_s((15)^*)$, $\pi_s((12)^*)$, and $\pi_s((1k)^*)$ are inscribed and by \Cref{lem:Miquel}, the four points \(\pi_s((1234)^*),\) \(\pi_s((134k)^*),\) \(\pi_s((1267)^*\}),\) \(\pi_s((167k)^*)\)  lie on a common plane.
This contradicts our previous observation $(\star)$ and thus \(1^*\) cannot be inscribed. 
\end{proof}

\begin{corollary}
Let $d\geq 4$ and $k\geq d+4$.
The cyclic polytope \(C_d(k)\) is not circumscribable.
\label{cor:all_cyclic}
\end{corollary}

\begin{proof}
The case $d=4$ is \Cref{thm:no_real_c48}.
Hence, assume $d=4+j$ with $j\geq 1$ and consider \(C_d(k)\) with \(k\geq d+4\). 
The vertex figure of vertex $k$ in $C_d(k)$ is combinatorially isomorphic to \(C_{d-1}(k-1)\).  
We iteratively take vertex figures of the largest labeled vertex $j$ times until we have \(C_4(k-j)\).
By \Cref{thm:no_real_c48}, \(C_4(k-j)\) is not circumscribable and, by Lemma~\ref{lem:classic_ii}, we conclude that \(C_d(k)\) is not circumscribable.
\end{proof}

\begin{table}[!htbp]
\footnotesize
\begin{tabular}{cc|cccccccc}
& & 		& 		  & 		& $k$	  & 		& 		   & 		  & \\[-3mm]
\wddots & \lcel{$C_d(k)$} & 3 		& 4 	  & 5		& 6	   	  & 7 		& 8 	   & 9 	      & \\\hhline{*{10}{-}}
\wddots & \lnum{2}		  & \bcheck & \bcheck & \bcheck & \bcheck & \bcheck & \bcheck  & \bcheck  & \bcdots \\ %\hline
\wddots & \lnum{3}        & 		& \bcheck & \bcheck & \bcheck & \bcheck & \bcheck  & \bcheck  & \bcdots \\ 
$d$ 	& \lnum{4}    	  & \wddots & 		  & \bcheck & \bcheck & \bcheck & \rtimes  & \rtimes  & \rcdots \\ 
\wddots & \lnum{5}        & 		& 		  & 		& \bcheck & \bcheck & ?		   & \rtimes  & \rcdots \\ 
\wddots	& \lnum{6}        & 		& 		  & 		& 		  & \bddots & \bddots  & $\ddots$ & \rddots \\ 
%\wddots	& 7 	   & 		 & 		   & 		 & 		   & 		 & \bddots  & \bddots  & $\ddots$ \\
\vspace{0.25cm}
\end{tabular}
\caption{Circumscribability of cyclic polytopes $C_d(k)$}
\label{tab:cyclic_poly}
\end{table}

Altogether, we have the following brief summary regarding the circumscribability of cyclic polytopes, see also \Cref{tab:cyclic_poly}:
\begin{itemize}
	\item Since polygons are circumscribable, \(C_2(k)\) is trivially circumscribable. 
	\item The cyclic polytope \(C_d(d+1)\) is combinatorially isomorphic to the $d$-simplex and hence circumscribable. 
	\item Similarly, \(C_d(d+2)\) is a direct sum of simplices and its dual is the product of simplices which is inscribable. Therefore $C_d(d+2)$ is circumscribable.
	\item The cyclic polytope \(C_3(k)=F_k^*\) is circumscribable, see \Cref{ssec:wedges}.
\end{itemize}
The only class of cyclic polytopes whose circumscribability is not determined is \(C_d(d+3)\).

\begin{question}
Is \(C_d(d+3)\) circumscribable for all \(d \geq 5\)?
\end{question}

In theory, this question can be addressed by the methods described in Sections \ref{ssec:47interpol} and \ref{ssec:47stereo}. 
Interpolation behaves interestingly: for $d=5$, we have $20$ facets and the space of quadrics in~$\P^5$ has dimension $21$. 
Hence, we expect a unique quadric containing all facet normals of a realization of $C_5(8)$. 
Computations suggest that the space of quadrics through these $20$ points is $3$-dimensional generically. 
Searching for a quadric with the right signature in this space is challenging. For larger $d$, the number of facets of $C_d(d+3)$ is bigger than the dimension of the space of quadrics. However, for $d=6$, the facet normals generically lie on a unique quadric. We did not manage to find one with the right signature. 
The computational approach via stereographic projection is already challenging for $C_4(7)$. 
For higher values of $d$, we are looking for Delaunay subdivisions of a $(d-1)$-dimensional simplex, another computational challenge.

%%%%%%%%%%%%%%%%%%%%%
\section{Forbidden subposet}
\label{sec:forbidden}
%%%%%%%%%%%%%%%%%%%%%

We present a combinatorial abstraction of the geometric obstruction presented in \Cref{ssec:48obstruction} using a poset. 

\begin{definition}[{Obstruction~$\mathcal{X}$}]
\label{def:subposet}
Let $P$ be a $d$-polytope.
We identity a hypothetical subposet~$\mathcal{X}$ of the face lattice of $P$ that creates an obstruction to inscribability.
This subposet consists of nine vertices $\{0,1,\dots,8\}$, two edges $\{12,34\}$, seven $2$-faces $\{A,B,C,D,E,X,Y\}$, and one $3$-face $\Phi$ that satisfy the following geometric properties in $P$. 

\begin{enumerate}[label=\roman{enumi}),ref=\ref{def:subposet}~\emph{\roman{enumi})}]
	\item The vertex $0$ has exactly $d$ neighboring vertices in $P$.
	\item The intersection of faces $X$ and $Y$ is the vertex $0$, which is not a vertex of $\Phi$.\label{def:subposet_ii} 
	\item $X$ contains the edge $12$.
	\item $Y$ contains the edge $34$.
	\item The $2$-faces $A,B,\dots, E$ are faces of $\Phi$.
	\item The $2$-faces $A,B,\dots,E$ contain the following vertices:
	\[
	\hspace{0.75cm}
	\{1,2,5,6\}\subseteq A,\ \{1,3,5,7\} \subseteq B,\ \{5,6,7,8\} \subseteq C,\ \{2,4,6,8\} \subseteq D,\ \{3,4,7,8\} \subseteq E.
	\]
\end{enumerate}
\end{definition}

\begin{remark}\ 
\begin{enumerate}[label=\alph{enumi})]
\item Since the $2$-faces $A$ and $B$ contain the vertices $1$ and $5$, $15$ must be an edge of $P$.
Similarly, $26$, $37$, $48$, $56$, $57$, $68$, $78$ are edges of $P$.
It follows that $C$ is a square and since $12$ and $34$ are edges of $P$ it follows that $A$, and $E$ are square faces too. 
\item Furthermore, by property \emph{ii)}, the face $\Phi$ does not contain $X$ nor $Y$.
\end{enumerate}
\end{remark}

See \Cref{fig:scheme_forbidden} for a scheme representing the five $2$-faces $A,\dots,E$ and \Cref{fig:forbidden} for an illustration of the Hasse diagram of $\mathcal{X}$.

\begin{figure}[!htbp]
\begin{tikzpicture}[%
	vertex/.style={inner sep=1pt,circle,draw=black!25!black,fill=black!75!black,thick}]

\def\innerdist{1}
\def\outerdist{2}

\coordinate (0) at (0,0);

\coordinate (a) at (\innerdist,0);
\coordinate (b) at (0,\innerdist);
\coordinate (c) at (-\innerdist,0);
\coordinate (d) at (0,-\innerdist);

\coordinate (A) at (\outerdist,0);
\coordinate (B) at (0,\outerdist);
\coordinate (C) at (-\outerdist,0);
\coordinate (D) at (0,-\outerdist);

\foreach \x/\y in {a/A,b/B,c/C,d/D}
  \draw[thick] (\x) -- (\y);

\draw[thick] (a) -- (b) -- (c) -- (d) -- cycle;
\draw[thick] (B) -- (C);
\draw[thick] (D) -- (A);

\draw[thick,out=60,in=30] (A) to (B); 
\draw[thick,out=240,in=210] (C) to (D); 

\node at (0,0) {$C$};
\node at ($0.75*(-\innerdist,\innerdist)$) {$A$};
\node at ($0.75*(\innerdist,-\innerdist)$) {$E$};
\node at ($0.75*(\innerdist,\innerdist)$) {$B$};
\node at ($0.75*(-\innerdist,-\innerdist)$) {$D$};

\node[vertex,label=left:$7$] at (a) {};
\node[vertex,label=below:$5$] at (b) {};
\node[vertex,label=right:$6$] at (c) {};
\node[vertex,label=above:$8$] at (d) {};
\node[vertex,label=right:$3$] at (A) {};
\node[vertex,label=above:$1$] at (B) {};
\node[vertex,label=left:$2$] at (C) {};
\node[vertex,label=below:$4$] at (D) {};

\node (circle5) [draw,dashed,red, thick,circle through=(A)] at (0,0) {};

\end{tikzpicture}
\caption{A schematization of the adjacencies between the $2$-faces $A,\dots,E$. The dashed circle is the circle obtained from Miquel's theorem.}
\label{fig:scheme_forbidden}
\end{figure}
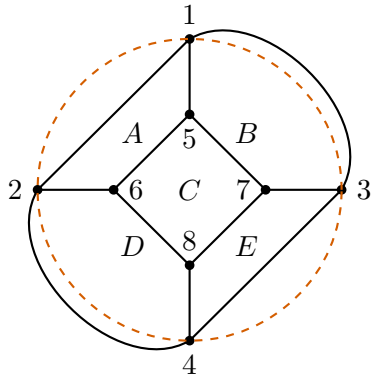

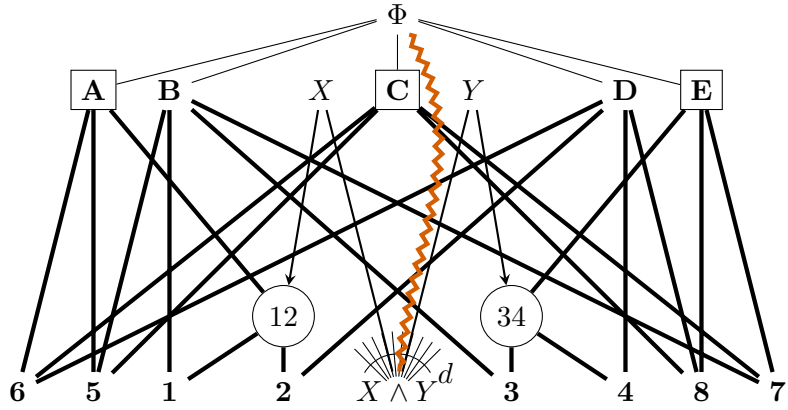
\begin{figure}[!htbp]
\begin{tikzpicture}[decoration = {zigzag,segment length = 2mm, amplitude = 0.5mm}]

\node at (0,-1) (9){$X\wedge Y$};

\foreach \angle in {45, 55, ..., 135}
  \draw[shorten <=0.2cm, shorten >= 0.2cm] (0,-1) -- +(\angle:1);
\draw ($(0,-1)+(30:0.5)$) arc (30:150:0.5);
\node at ($(0,-1)+(30:0.5)+(0.2,0)$) {$d$};

\node at (-3,-1) (1){\bf 1};
\node at (-1.5,-1) (2){\bf 2};
\node at (-4,-1) (5){\bf 5};
\node at (-5,-1) (6){\bf 6};
\node at (5,-1) (7){\bf 7};
\node at (4,-1) (8){\bf 8};
\node at (1.5,-1) (3){\bf 3};
\node at (3,-1) (4){\bf 4};

\node[circle,draw=black] at (-1.5,0) (12){12};
\node[circle,draw=black] at (1.5,0) (34){34};

\node at (-1,3) (ab){$X$};
\node[rectangle,draw=black] at (-4,3) (bf){$\bf A$};
\node at (-3,3) (hf){$\bf B$};
\node[rectangle,draw=black]  at (0,3) (gf){$\bf C$};
\node at (3,3) (if){$\bf D$};
\node[rectangle,draw=black]  at (4,3) (cf){$\bf E$};
\node at (1,3) (cd){$Y$};

\node at (0,4) (f){$\Phi$};

%\draw (3) -- (cf) -- (4) -- (df) -- (3) -- (cd) -- (4);
%\draw (2) -- (ab) -- (1) -- (af) -- (2) -- (bf) -- (1);
\draw[ultra thick] (2) -- (12) -- (1);
\draw[<-,thick,>=stealth] (12) -- (ab);
\draw[ultra thick] (12) -- (bf);
\draw[ultra thick] (3) -- (34) -- (4);
\draw[ultra thick] (cf) -- (34);
\draw[->,thick,>=stealth] (cd) -- (34);

\draw[ultra thick] (5) -- (bf) -- (6) -- (gf) -- (5) -- (hf) -- (1);
\draw[ultra thick] (7) -- (gf) -- (8) -- (cf) -- (7) -- (hf) -- (3);
\draw[ultra thick] (4) -- (if) -- (8);
\draw[ultra thick] (6) -- (if) -- (2);
\draw (f) -- (cf);
\draw (f) -- (bf);
\draw (if) -- (f) -- (hf);
\draw[thick] (ab) -- (9) -- (cd);
%\draw (ab) to[out=210, in=180] (9);
%\draw (9) to[in=300,out=0] (cd);
\draw (gf) -- (f);
% \draw[fill=black, opacity=0.4] (9) -- (0.7,0.5) -- (-0.7,0.5) -- (9);
% \node at (0,0.6) {$d$};
\draw[ultra thick,color=red,decorate] (f) to[out=300, in=85] (9);

\end{tikzpicture}
\caption{An inscribability obstruction poset $\mathcal{X}$. The zig-zag edge represents a required non-relation.}
\label{fig:forbidden}
\end{figure}

Assuming that $\Phi$ is inscribed, Miquel's theorem implies that the edges $12$ and $34$ are coplanar.
Since $X$ and $Y$ are two $2$-faces intersecting in exactly one vertex $0$ with exactly $d$ neighbors, the edges $12$ contained in $X$ and the edge $34$ contained in~$Y$ must be skew.
Since the edges $12$ and $34$ cannot be simultaneously coplanar and skew, we obtain the following obstruction lemma.

\begin{lemma}[Obstruction Lemma]
\label{lem:obstruction}
Let $P$ be a $d$-polytope.
If the face lattice of $P$ admits $\mathcal{X}$ as a subposet with the properties:
\begin{itemize}
	\item[T)] The meet $X\wedge Y$ in the face lattice of $P$ is $0$, (Touching)
	\item[S)] $0$ has exactly $d$ covers in the face lattice of $P$, (Simple)
\end{itemize}
then $P$ has no realization where the face $\Phi$ is inscribed. \hfill $\square$
\end{lemma}

\begin{algorithm}[!htbp]
\small
\begin{flushleft}
\textbf{Input:} A combinatorial type of polytope $P$\\
\textbf{Output:} Either finds a Miquel's polytope $\Phi$ and two $2$-faces $X,Y$ or shows that it satisfies the necessary condition.
\end{flushleft}
\caption{Checking for obstruction $\mathcal{X}$ in the face lattice $\Lambda_P$ of a $d$-polytope $P$}\label{algo:obstruction}
\begin{algorithmic}[1]
\Procedure{FindObstruction}{$P$}\Comment{Tests the presence of $\mathcal{X}$ in $\Lambda_P$}
\State $P_3:=\{f\in \Lambda_P~:~ \dim f = 3\}$
\State $Q\gets$ $3$-skeleton of $P$\Comment{Makes incidence verification linear}
\State Found $\gets {\tt False}$
\While{$\neg$ Found and $|P_3|>0$}\Comment{$O(k^4)$}
\State $\Phi\gets$ an element of $P_3$
\State $P_3\gets P_3\setminus\{\Phi\}$
\For{each square configuration $A,C,E$ in $\Phi$}\Comment{$O(k^3)$}
\If{$\Phi$ contains faces $B$ and $D$}
\For{$X$ cover of $12$, and $Y$ cover of $34$}\Comment{$O(k^2)$}
\If{$0:=X\wedge Y$ is a simple vertex and $0\not\in\Phi$}
\State Found $\gets {\tt True}$
\EndIf
\EndFor
\EndIf
\EndFor
\EndWhile\label{euclidendwhile}
\If{Found}
\textbf{return} $(\text{Found},\Phi,X,Y)$ \Comment{The obstruction was found.}
\Else
\textbf{ return} $\tt None$ \Comment{The necessary condition is fulfilled.}
\EndIf
\EndProcedure

\end{algorithmic}
\label{alg}
\end{algorithm}

\Cref{alg} uses Obstruction~$\mathcal{X}$ to detect non-inscribability.
It works in any dimension $d\geq4$ and only requires the $3$-skeleton of the polytope.
On the one hand, the algorithm can be generalized to obstructions obtained from other planar ``Delaunay'' circle theorems and to larger face figures.
On the other hand, it only provides a necessary condition for a combinatorial type of polytope to be inscribable.
A naive implementation of \Cref{alg} leads to a running time of $O(k^9)$, where $k$ is the number of vertices of the polytope.

Running this algorithm on the \(8\)-facet polytopes results in a combinatorial type with \(f\)-vector \((14,31,25,8)\) which is not inscribable because it contains \(\mathcal{X}\) and has the Simple and Touching properties.
The facets of this combinatorial type are
\begin{align*}
&0126ABC, 0159BCD, 02367ACD, 04589ABD,\\
&123456789, 12345AB, 16789CD, 3478AD.
\end{align*}

The illustration of the stereographic projection from vertex $0$ of this polytope in Figure~\ref{fig:small_non} shows that this is the smallest with \(8\) facets; contracting any face destroys some critical component of the obstruction.

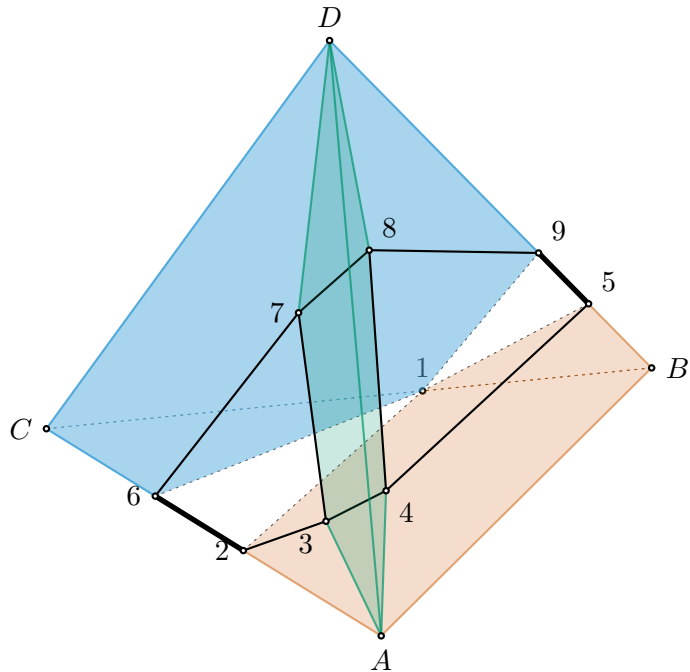
\begin{figure}[!htbp]
\begin{center}
\begin{tikzpicture}%
	[x={(-0.084911cm, 0.986684cm)},
	y={(0.599689cm, -0.060579cm)},
	z={(-0.795715cm, -0.150945cm)},
	scale=4.000000,
	back/.style={dotted, thin,opacity=0.7},
	edge/.style={color=black, thick, line cap=round},
	facet/.style={fill=white,fill opacity=0},
	edge3/.style={color=red!95!black, thick, line cap=round, opacity=0.500000},
	facet3/.style={fill=red!95!black,fill opacity=0.200000},
	edge2/.style={color=skyblue!95!black, thick, line cap=round},
	facet2/.style={fill=skyblue!95!black,fill opacity=0.500000},
	edge1/.style={color=bluishgreen!95!black, thick, line cap=round, opacity=0.700000},
	facet1/.style={fill=bluishgreen!95!black,fill opacity=0.200000},
	vertex/.style={inner sep=0.025cm,circle,draw=black,fill=white,thick},
	decoration={snake,amplitude=.3mm,segment length=0.75mm}]
%
%
%% Coordinate of the vertices:
%%

\coordinate (0.00000, 0.00000, 0.00000) at (0.00000, 0.00000, 0.00000);
\coordinate (2.00000, 0.00000, 0.00000) at (2.00000, 0.00000, 0.00000);
\coordinate (1.00000, 1.62500, 0.00000) at (1.00000, 1.62500, 0.00000);
\coordinate (1.00000, 0.61538, 1.74109) at (1.00000, 0.61538, 1.74109);

\coordinate (1.19580, 1.30682, 0.00000) at (1.19580, 1.30682, 0.00000);
\coordinate (0.49354, 0.09658, 0.00000) at (0.49354, 0.09658, 0.00000);
\coordinate (0.42857, 0.08791, 0.24873) at (0.42857, 0.08791, 0.24873);
\coordinate (1.35110, 1.05447, 0.00000) at (1.35110, 1.05447, 0.00000);
\coordinate (0.41176, 0.25339, 0.71692) at (0.41176, 0.25339, 0.71692);
\coordinate (1.30299, 0.11895, 0.00000) at (1.30299, 0.11895, 0.00000);
\coordinate (1.13793, 0.10610, 0.30019) at (1.13793, 0.10610, 0.30019);
\coordinate (0.67480, 0.41526, 1.17488) at (0.67480, 0.41526, 1.17488);
\coordinate (1.00000, 1.24298, 0.65879) at (1.00000, 1.24298, 0.65879);

%% Drawing edges in the back
%%
\draw[edge,back] (1.00000, 1.62500, 0.00000) -- (1.00000, 0.61538, 1.74109);

\draw[edge3,back] (1.00000, 1.62500, 0.00000) -- (1.00000, 1.24298, 0.65879);

\draw[edge2,back] (1.00000, 0.61538, 1.74109) -- (1.00000, 1.24298, 0.65879);

\draw[edge,back] (1.19580, 1.30682, 0.00000) -- (1.00000, 1.24298, 0.65879);
\draw[edge,back] (1.35110, 1.05447, 0.00000) -- (1.00000, 1.24298, 0.65879);
\draw[edge,back] (0.41176, 0.25339, 0.71692) -- (1.00000, 1.24298, 0.65879);
\draw[edge,back] (0.67480, 0.41526, 1.17488) -- (1.00000, 1.24298, 0.65879);
%%
%%
%% Drawing vertices in the back
%%
%%

\node[vertex,label=above:$1$] at (1.00000, 1.24298, 0.65879)     {};
%%
%% Drawing the facets
%%
\fill[facet] (1.00000, 1.62500, 0.00000) -- (0.00000, 0.00000, 0.00000) -- (2.00000, 0.00000, 0.00000) -- cycle {};
\fill[facet] (1.00000, 0.61538, 1.74109) -- (0.00000, 0.00000, 0.00000) -- (2.00000, 0.00000, 0.00000) -- cycle {};

\fill[facet3] (1.00000, 1.24298, 0.65879) -- (1.19580, 1.30682, 0.00000) -- (0.49354, 0.09658, 0.00000) -- (0.42857, 0.08791, 0.24873) -- (0.41176, 0.25339, 0.71692) -- cycle {};
\fill[facet3] (0.49354, 0.09658, 0.00000) -- (0.00000, 0.00000, 0.00000) -- (1.00000, 1.62500, 0.00000) -- (1.19580, 1.30682, 0.00000) -- cycle {};
\fill[facet3] (0.42857, 0.08791, 0.24873) -- (0.00000, 0.00000, 0.00000) -- (0.49354, 0.09658, 0.00000) -- cycle {};
\fill[facet3] (0.41176, 0.25339, 0.71692) -- (0.00000, 0.00000, 0.00000) -- (0.42857, 0.08791, 0.24873) -- cycle {};

\fill[facet2] (1.00000, 1.24298, 0.65879) -- (1.35110, 1.05447, 0.00000) -- (1.30299, 0.11895, 0.00000) -- (1.13793, 0.10610, 0.30019) -- (0.67480, 0.41526, 1.17488) -- cycle {};
\fill[facet2] (1.13793, 0.10610, 0.30019) -- (2.00000, 0.00000, 0.00000) -- (1.30299, 0.11895, 0.00000) -- cycle {};
\fill[facet2] (0.67480, 0.41526, 1.17488) -- (1.00000, 0.61538, 1.74109) -- (2.00000, 0.00000, 0.00000) -- (1.13793, 0.10610, 0.30019) -- cycle {};
\fill[facet2] (1.30299, 0.11895, 0.00000) -- (2.00000, 0.00000, 0.00000) -- (1.35110, 1.05447, 0.00000) -- cycle {};

%%
%%
%% Drawing edges in the front
%%

\draw[edge3] (0.00000, 0.00000, 0.00000) -- (1.00000, 1.62500, 0.00000);
\draw[edge1] (0.00000, 0.00000, 0.00000) -- (0.49354, 0.09658, 0.00000);
\draw[edge1] (0.00000, 0.00000, 0.00000) -- (0.42857, 0.08791, 0.24873);
\draw[edge3] (0.00000, 0.00000, 0.00000) -- (0.41176, 0.25339, 0.71692);
\draw[edge3] (1.00000, 1.62500, 0.00000) -- (1.19580, 1.30682, 0.00000);
\draw[edge] (1.19580, 1.30682, 0.00000) -- (0.49354, 0.09658, 0.00000);
\draw[edge] (0.49354, 0.09658, 0.00000) -- (0.42857, 0.08791, 0.24873);
\draw[edge] (0.42857, 0.08791, 0.24873) -- (0.41176, 0.25339, 0.71692);

\draw[edge2] (2.00000, 0.00000, 0.00000) -- (1.00000, 0.61538, 1.74109);
\draw[edge2] (2.00000, 0.00000, 0.00000) -- (1.35110, 1.05447, 0.00000);
\draw[edge1] (2.00000, 0.00000, 0.00000) -- (1.30299, 0.11895, 0.00000);
\draw[edge1] (2.00000, 0.00000, 0.00000) -- (1.13793, 0.10610, 0.30019);
\draw[edge2] (1.00000, 0.61538, 1.74109) -- (0.67480, 0.41526, 1.17488);
\draw[edge] (1.35110, 1.05447, 0.00000) -- (1.30299, 0.11895, 0.00000);
\draw[edge] (1.30299, 0.11895, 0.00000) -- (1.13793, 0.10610, 0.30019);
\draw[edge] (1.13793, 0.10610, 0.30019) -- (0.67480, 0.41526, 1.17488);

\fill[facet1] (1.13793, 0.10610, 0.30019) -- (2.00000, 0.00000, 0.00000) -- (0.00000, 0.00000, 0.00000) -- (0.42857, 0.08791, 0.24873) -- cycle {};
\fill[facet1] (1.30299, 0.11895, 0.00000) -- (2.00000, 0.00000, 0.00000) -- (0.00000, 0.00000, 0.00000) -- (0.49354, 0.09658, 0.00000) -- cycle {};

\draw[edge,line width=2pt] (1.19580, 1.30682, 0.00000) -- (1.35110, 1.05447, 0.00000);
\draw[edge] (0.49354, 0.09658, 0.00000) -- (1.30299, 0.11895, 0.00000);
\draw[edge] (0.42857, 0.08791, 0.24873) -- (1.13793, 0.10610, 0.30019);
\draw[edge,line width=2pt] (0.41176, 0.25339, 0.71692) -- (0.67480, 0.41526, 1.17488);

\draw[edge1] (0.00000, 0.00000, 0.00000) -- (2.00000, 0.00000, 0.00000);
%%
%%
%% Drawing the vertices in the front
%%
\node[vertex,label=below:$A$] at (0.00000, 0.00000, 0.00000)     {};
\node[vertex,label=above:$D$] at (2.00000, 0.00000, 0.00000)     {};
\node[vertex,label=right:$B$] at (1.00000, 1.62500, 0.00000)     {};
\node[vertex, label=left:$C$] at (1.00000, 0.61538, 1.74109)     {};

\node[vertex,label=above right:$5$] at (1.19580, 1.30682, 0.00000)     {};
\node[vertex,label=below right:$4$] at (0.49354, 0.09658, 0.00000)     {};
\node[vertex,label=below left:$3$] at (0.42857, 0.08791, 0.24873)     {};
\node[vertex,label=above right:$9$] at (1.35110, 1.05447, 0.00000)     {};
\node[vertex, label=left:$2$] at (0.41176, 0.25339, 0.71692)     {};
\node[vertex,label=above right:$8$] at (1.30299, 0.11895, 0.00000)     {};
\node[vertex, label=left:$7$] at (1.13793, 0.10610, 0.30019)     {};
\node[vertex, label=left:$6$] at (0.67480, 0.41526, 1.17488)     {};

\end{tikzpicture}
\end{center}
\caption{The smallest polytope with 8 facets that contains the obstruction~\(\mathcal{X}\) with the Simple and Touching properties from \Cref{lem:obstruction}}
\label{fig:small_non}
\end{figure}

%%%%%%%%%%%%%%%%%%%%%
%%%%%%%%%%%%%%%%%%%%%
\section{The neighborly $4$-polytopes with $8$ vertices $\mathcal{N}_4(8)$}
\label{sec:neighborly}
%%%%%%%%%%%%%%%%%%%%%
%%%%%%%%%%%%%%%%%%%%%

In the previous sections, we identified a combinatorial barrier to inscribability.
We use this barrier, and a slight generalization of it to prove the following theorem.

\begin{theorem}
No polytope with $f$-vector $(8,28,40,20)$ is circumscribable. 
Dually, no polytope with $f$-vector $(20,40,28,8)$ is inscribable.
\label{thm:n48}
\end{theorem}

\begin{proof}
There are three combinatorial types of polytope with the given $f$-vector \cite{gruenbaum_enumeration_1967}. 

\textbf{Case~1.} The first type is the cyclic polytope \(C_4(8)\), see~\Cref{cor:c48}.

\textbf{Case 2.} Consider the combinatorial type $N_4^2(8)$ given by the facet-vertex incidences below. We denote the vertices from $v_1$ to $v_8$. 
The numbers 0-9 and letters $A$ to $J$ denote facets.
\[
\begin{tabular}{l@{\hspace{0.5cm}}l}
$v_1: \{0,1,2,3,4,5,6,7,8,9\}$ (blue, bottom)  & $v_5: \{1,4,6,8,9,A,B,F,I,J\}$ (triangle ``18I'') \\
$v_2: \{2,3,4,5,6,7,A,B,C,D\}$ (white, center) & $v_6: \{0,1,2,4,A,E,F,H,I,J\}$ (triangle ``1HI'') \\
$v_3: \{0,2,3,A,B,C,D,E,F,G\}$ (orange, left)   & $v_7: \{0,1,3,5,8,C,E,G,H,J\}$ (triangle ``18H'') \\
$v_4: \{6,7,9,B,D,E,F,G,H,I\}$ (green, top) & $v_8: \{5,7,8,9,C,D,G,H,I,J\}$ (triangle ``8HI'') 
\end{tabular}
\]
\begin{figure}[!htbp]
\input{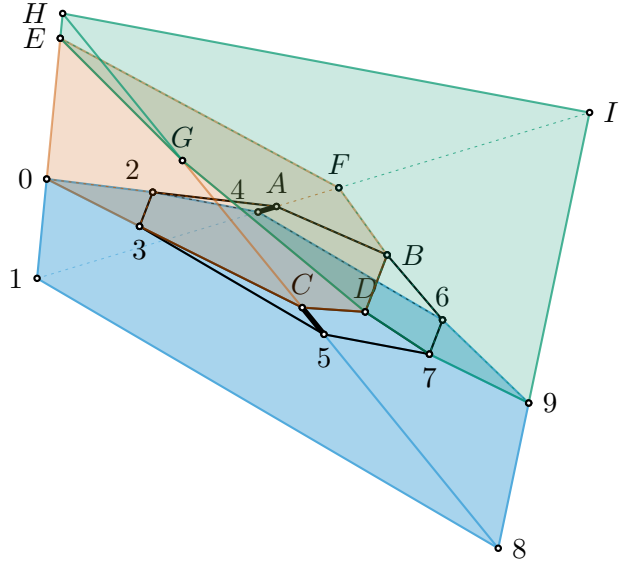}
\caption{The image of the stereographic projection of $N_4^2(8)$ from vertex~$J$. The two bold edges $4A$ and $5C$ in the wedge $v_2$ must be coplanar by Miquel's theorem.}
\label{fig:n48_2}
\end{figure}

Projecting $N_4^2(8)^*$ stereographically from vertex $J$, we obtain a subdivision of a tetrahedron as illustrated in Figure~\ref{fig:n48_2}.
Assuming that the wedge $v_2$ is inscribed, this implies that quadruples $(A,B,C,D)$, $(4,5,6,7)$, $(4,6,A,B)$, $(6,7,B,D)$, and $(5,7,C,D)$ are all coplanar and lie on a sphere. 
By \Cref{lem:Miquel}, this implies that quadruple $(4,5,A,C)$ is coplanar, and thus $4A$ and $5C$ are coplanar as well.
Edge~$4A$ belongs to wedges $v_2$, $v_5$ and $v_6$, while edge $5C$ belongs to wedges $v_2$, $v_7$, and~$v_8$.
Since $N_4^2(8)^*$ is convex, edges $4A$ and $5C$ must be skew, since they belong to two $2$-faces intersecting in $J$, see \Cref{fig:n48_2}.
Therefore, if $N_4^2(8)$ is convex, facet $v_2$ cannot be inscribed and $N_4^2(8)^*$ is not inscribable. 
Hence, $N_4^2(8)$ cannot be circumscribable by Lemma~\ref{lem:classic}.
Facets $v_2$ and $v_8$ are combinatorially equivalent in $N_4^2(8)^*$ and hence both cannot be inscribed; the problematic pair of edges in $v_8$ is $CD/IJ$.

\textbf{Case 3.}
The final combinatorial type $N_4^3(8)$ is determined by the following facet-vertex incidences:
\[
\begin{tabular}{l@{\hspace{0.5cm}}l}
$v_1: \{0,1,2,3,4,5,6,7,8,9\}$ (orange, back left)    &  $v_5: \{0,1,3,4,7,8,A,E,G,J\}$ (triangle ``3AE'') \\
$v_2: \{1,2,5,6,8,9,A,B,C,D\}$ (green, back right)  &  $v_6: \{3,5,8,A,B,C,D,H,I,J\}$ (triangle ``3AI'') \\
$v_3: \{0,2,7,9,B,C,E,F,G,H\}$ (blue, front bottom)   &  $v_7: \{0,1,2,A,B,E,F,H,I,J\}$ (triangle ``AEI'') \\
$v_4: \{4,6,7,9,C,D,F,G,H,I\}$ (white, front top)  &  $v_8: \{3,4,5,6,D,E,F,G,I,J\}$ (triangle ``3EI'') \\
\end{tabular}
\]
\begin{figure}[!htbp]
\input{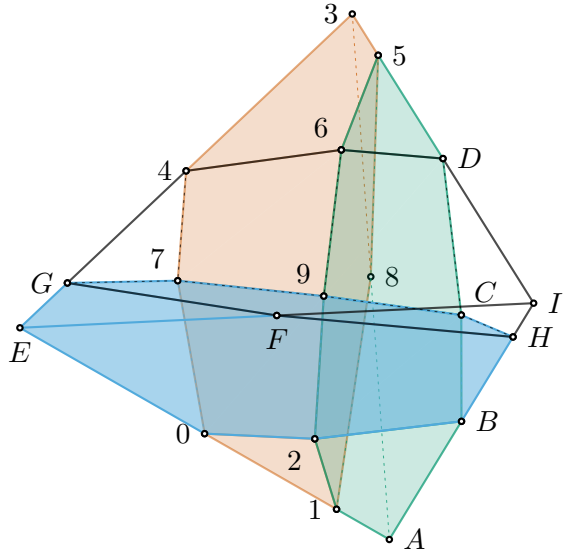}
\caption{The image of the stereographic projection of $N_4^3(8)$ from vertex~$J$.}
\label{fig:n48_3}
\end{figure}

Projecting $N_4^3(8)^*$ stereographically from vertex $J$, we obtain a subdivision of a tetrahedron as illustrated in Figure~\ref{fig:n48_3}.
Assuming that wedge $v_1$ is inscribed, it follows that the quadruples $(0,2,7,9)$, $(4,6,7,9)$, $(3,4,5,6)$, $(2,5,6,9)$, and $(0,3,4,7)$ are coplanar and the eight points lie on a sphere. 
By \Cref{lem:Miquel}, the quadruple $(0,2,3,5)$ must be coplanar, and thus $02$ and $35$ are also coplanar.
Now, consider the hexagon $02BEFH$ of wedge $v_3$.
Because the hexagon is convex, the line spanned by edge $02$ intersects edge $AI$ strictly between the point~$A$ and~$B$.
But $02$ and $35$ are coplanar, and since the line spanned by $02$ meets both the lines spanned by $AI$ and $35$, they must meet in $I$.
This forces the points $B$, $H$, and $I$ to collapse, a contradiction.
Hence $N_4^3(8)$ is not circumscribable by Lemmas~\ref{lem:classic} and~\ref{lem:projection}.
The facets $v_1,v_2,v_3$, and $v_4$ are combinatorially equivalent in $N_4^3(8)^*$. Hence, none of them can be inscribed.
\end{proof}

%%%%%%%%%%%%%%%%%%%%%
\section{Open Questions}
\label{sec:questions}
%%%%%%%%%%%%%%%%%%%%%

The previous sections provide some concrete support for \cite[Conjecture 8.4]{chen_scribability_2017}, that all large neighborly polytopes are not circumscribable. 
The following approach may yield a rich infinite class of not circumscribable neighborly polytopes.

Starting from the basepoint of \(C_d(k)\), with \(k> d+4\) and \(d>4\), we have one neighborly polytope for each pair $(k,d)$ that is not cicrcumscribable. 
%These polytopes are in fact far from being inscribable, as no facet can be inscribed.
From some neighborly polytopes, adding a single vertex yeilds another neighborly polytope. Iterating this process can give rise to many neighborly polytopes.
This is described in detail in \cite{padrol_many_2013}.  Dually, we may generate dual to neighborly polytopes by introducing a single new facet. 
For an example of this operation, compare \Cref{fig:c47_facets} and \Cref{fig:n48_2}: 
The facet \(6^*\) is split into two facets, the facet containing edge \(E0\) on the left, and the facet outlined in black. 
Once this operation is done, most of the squares are split into further squares. 
Having many squares in a common facet can lead to an obstruction to inscribability.
In particular, this sequence of three squares in a row described in  \Cref{sec:forbidden} is disadvantageous to inscribing a polytope.
There are neighborly $4$-polytopes with $9$ vertices that do not have Miquel's structure as in the statement of \Cref{lem:Miquel} as a facet.
We ask a more specific version of \cite[Conjecture~8.4]{chen_scribability_2017}:

\begin{question}
Are neighborly polytopes avoiding Miquel's arrangement in vertex figures circumscribable?
\end{question}

We restate the open questions brought up throughout the text.

\begin{question}
For which \(d\) is the polytope \(C_d(d+3)\) circumscribable?
\end{question}

This question has an obvious line of attack: Gale duality. 
Depending on the particular choice of reductions in the duality, understanding an alternating sequence of black and white dots on a line explains the general case.
If \(C_d(d+3)\) is not circumscribable for some \(d \geq 5\), it would constitute a counterexample to a conjecture raised by Gr\"unbaum \cite[Last sentence of Section 3.15]{gruenbaum_convex_1967}

Our final question has to do with inscribed realizations of \(C_4(7)^*\).
We gave two ways to see that it is inscribable, the second of which gives explicit coordinates.
However, the coordinates are in a degree twenty extension of \(\mathbb{Q}\).
We wonder what is the smallest degree extension needed to inscribe \(C_4(7)^*\).
In particular,

\begin{question}
Is \(C_4(7)^*\) inscribable with rational coordinates?
\end{question}

\bibliographystyle{myamsalpha}
\bibliography{bibliography}

\newpage
\appendix
\section{Inscribed realization of small $f$-vectors}
\label{appendix:inscribed_coordinates}

The following coordinates give rational inscribed realizations for polytopes with $f_0\in\{6,7\}$, with $f_0=8$ and $f_3=7$, and for two types with $f_0=9$ and $f_3=7$.\\

\resizebox{\textwidth}{!}{
\begin{tabular}{cl}
$f$-vector & Coordinates \\\hline
(6, 13, 13, 6)  & $((0, 0, 0, 0), (0, 0, 0, 1), (0, 0, 1, 0), (0, 1, 0, 0), (1, 0, 0, 0), (1, 0, 1, 0))$ \\[0.25em]
(6, 14, 15, 7)  & $((0, 0, 0, 0), (0, 0, 0, 1), (0, 0, 1, 0), (0, 1, 0, 0), (1, 0, 0, 0), (1, 1, 1, 0))$ \\[0.25em]
(6, 14, 16, 8)  & $((-1, 0, 0, 0), (0, -\frac{1}{3}, -\frac{2}{3}, -\frac{2}{3}), (0, 0, 0, 1), (0, 0, 1, 0), (0, 1, 0, 0), (1, 0, 0, 0))$ \\[0.25em]
(6, 15, 18, 9)  & $((-\frac{3}{5}, -\frac{4}{5}, 0, 0), (0, 0, -\frac{3}{5}, -\frac{4}{5}), (0, 0, 0, 1), (0, 0, 1, 0), (0, 1, 0, 0), (1, 0, 0, 0))$ \\[0.25em]
(7, 15, 14, 6)  & $((0, 0, 0, 0), (0, 0, 0, 1), (0, 0, 1, 0), (0, 1, 0, 0), (1, 0, 0, 0), (1, 0, 1, 0), (1, 1, 0, 0))$ \\[0.25em]
(7, 16, 16, 7)  & $((0, 0, -1, 0), (0, 0, 0, -1), (0, 0, 0, 1), (0, 0, \frac{3}{5}, \frac{4}{5}), (0, 0, 1, 0), (0, 1, 0, 0), (1, 0, 0, 0))$ \\[0.25em]
(7, 16, 16, 7)  & $((0, -\frac{4}{5}, -\frac{3}{5}, 0), (0, 0, -1, 0), (0, 0, 0, 1), (0, 0, 1, 0), (0, 1, 0, 0), (\frac{6}{7}, \frac{3}{7}, \frac{2}{7}, 0), (1, 0, 0, 0))$ \\[0.25em]
(7, 17, 17, 7)  & $((0, 0, 0, 0), (0, 0, 0, 1), (0, 0, 1, 0), (0, 1, 0, 0), (1, 0, 0, 0), (1, 1, 0, 1), (1, 1, 1, 0))$ \\[0.25em]
(7, 17, 18, 8)  & $((0, 0, 0, 0), (0, 0, 0, 1), (0, 0, 1, 0), (0, 1, 0, 0), (0, 1, 0, 1), (1, 0, 0, 0), (1, 0, 1, 0))$ \\[0.25em]
(7, 17, 18, 8)  & $((-\frac{3}{5}, 0, 0, \frac{4}{5}), (0, 0, -1, 0), (0, 0, 0, 1), (0, 0, 1, 0), (0, 1, 0, 0), (\frac{6}{7}, \frac{3}{7}, \frac{2}{7}, 0), (1, 0, 0, 0))$ \\[0.25em]
(7, 17, 18, 8)  & $((0, 0, -1, 0), (0, 0, 0, 1), (0, 0, 1, 0), (0, 1, 0, 0), (\frac{3}{13}, \frac{12}{13}, \frac{4}{13}, 0), (\frac{4}{5}, 0, -\frac{3}{5}, 0), (1, 0, 0, 0))$ \\[0.25em]
(7, 17, 18, 8)  & $((-\frac{12}{13}, -\frac{5}{13}, 0, 0), (0, 0, 0, 1), (0, 0, 1, 0), (0, 1, 0, 0), (\frac{4}{9}, \frac{7}{9}, \frac{4}{9}, 0), (\frac{20}{29}, -\frac{21}{29}, 0, 0), (1, 0, 0, 0))$ \\[0.25em]
(7, 18, 19, 8)  & $((0, -\frac{2}{7}, -\frac{6}{7}, \frac{3}{7}), (0, 0, 0, 1), (0, 0, 1, 0), (0, \frac{2}{7}, \frac{6}{7}, \frac{3}{7}), (0, 1, 0, 0), (\frac{1}{2}, \frac{1}{2}, \frac{1}{2}, \frac{1}{2}), (1, 0, 0, 0))$ \\[0.25em]
(7, 17, 19, 9)  & $((-\frac{4}{5}, 0, -\frac{3}{5}, 0), (0, 0, 0, 1), (0, 0, 1, 0), (0, 1, 0, 0), (\frac{3}{5}, 0, \frac{4}{5}, 0), (\frac{14}{15}, \frac{4}{15}, \frac{1}{5}, \frac{2}{15}), (1, 0, 0, 0))$ \\[0.25em]
(7, 18, 20, 9)  & $((-\frac{4}{5}, -\frac{3}{5}, 0, 0), (0, 0, 0, 1), (0, 0, 1, 0), (0, 1, 0, 0), (\frac{6}{11}, \frac{6}{11}, \frac{7}{11}, 0), (\frac{7}{9}, -\frac{4}{9}, -\frac{4}{9}, 0), (1, 0, 0, 0))$ \\[0.25em]
(7, 18, 20, 9)  & $((-\frac{6}{7}, \frac{2}{7}, \frac{3}{7}, 0), (0, -\frac{3}{5}, -\frac{4}{5}, 0), (0, 0, 0, 1), (0, 0, 1, 0), (0, 1, 0, 0), (\frac{2}{7}, \frac{3}{7}, \frac{6}{7}, 0), (1, 0, 0, 0))$ \\[0.25em]
(7, 18, 20, 9)  & $((0, -\frac{7}{25}, -\frac{24}{25}, 0), (0, 0, 0, 1), (0, 0, 1, 0), (0, 1, 0, 0), (\frac{5}{9}, \frac{2}{3}, \frac{4}{9}, \frac{2}{9}), (\frac{12}{13}, \frac{4}{13}, \frac{3}{13}, 0), (1, 0, 0, 0))$ \\[0.25em]
(7, 18, 20, 9)  & $((-\frac{4}{5}, 0, -\frac{3}{5}, 0), (0, 0, 0, 1), (0, 0, 1, 0), (0, 1, 0, 0), (\frac{7}{9}, 0, \frac{4}{9}, \frac{4}{9}), (\frac{6}{7}, \frac{2}{7}, \frac{3}{7}, 0), (1, 0, 0, 0))$ \\[0.25em]
(7, 18, 20, 9)  & $((-\frac{4}{5}, 0, -\frac{3}{5}, 0), (-\frac{2}{7}, 0, \frac{3}{7}, \frac{6}{7}), (0, 0, 0, 1), (0, 0, 1, 0), (0, 1, 0, 0), (\frac{4}{13}, \frac{3}{13}, \frac{12}{13}, 0), (1, 0, 0, 0))$ \\[0.25em]
(7, 18, 20, 9)  & $((0, 0, -1, 0), (0, 0, -\frac{3}{5}, \frac{4}{5}), (0, 0, 0, 1), (0, 0, 1, 0), (0, 1, 0, 0), (\frac{4}{9}, \frac{7}{9}, \frac{4}{9}, 0), (1, 0, 0, 0))$ \\[0.25em]
(7, 18, 21, 10) & $((0, 0, 0, 0), (0, 0, 0, 1), (0, 0, 1, 0), (0, 1, 0, 0), (1, 0, 0, 0), (1, 1, 1, 0), (1, 1, 1, 1))$ \\[0.25em]
(7, 18, 21, 10) & $((-\frac{10}{11}, -\frac{2}{11}, \frac{4}{11}, -\frac{1}{11}), (-\frac{3}{5}, 0, -\frac{4}{5}, 0), (0, 0, 0, 1), (0, 0, 1, 0), (0, 1, 0, 0), (\frac{2}{7}, \frac{3}{7}, \frac{6}{7}, 0), (1, 0, 0, 0))$ \\[0.25em]
(7, 18, 21, 10) & $((0, 0, 0, 0), (0, 0, 0, 1), (0, 0, 1, 0), (0, 1, 0, 0), (1, 0, 0, 0), (1, 0, 1, 0), (1, 1, 1, 1))$ \\[0.25em]
(7, 18, 21, 10) & $((-\frac{4}{5}, -\frac{3}{5}, 0, 0), (0, 0, 0, 1), (0, 0, 1, 0), (0, 1, 0, 0), (\frac{3}{13}, \frac{12}{13}, \frac{4}{13}, 0), (\frac{6}{7}, \frac{3}{7}, 0, -\frac{2}{7}), (1, 0, 0, 0))$ \\[0.25em]
(7, 19, 22, 10) & $((-\frac{12}{13}, 0, -\frac{5}{13}, 0), (-\frac{6}{11}, 0, -\frac{2}{11}, \frac{9}{11}), (0, 0, 0, 1), (0, 0, 1, 0), (0, 1, 0, 0), (\frac{1}{3}, \frac{2}{3}, \frac{2}{3}, 0), (1, 0, 0, 0))$ \\[0.25em]
(7, 19, 22, 10) & $((-\frac{4}{5}, -\frac{3}{5}, 0, 0), (0, 0, 0, 1), (0, 0, 1, 0), (0, 1, 0, 0), (\frac{6}{11}, -\frac{6}{11}, 0, \frac{7}{11}), (\frac{6}{7}, \frac{3}{7}, \frac{2}{7}, 0), (1, 0, 0, 0))$ \\[0.25em]
(7, 18, 22, 11) & $((-\frac{4}{5}, 0, -\frac{3}{5}, 0), (0, 0, 0, 1), (0, 0, 1, 0), (0, 1, 0, 0), (\frac{2}{3}, \frac{2}{3}, \frac{1}{3}, 0), (\frac{12}{17}, 0, -\frac{12}{17}, -\frac{1}{17}), (\frac{112}{113}, \frac{15}{113}, 0, 0))$ \\[0.25em]
(7, 19, 23, 11) & $((-\frac{4}{5}, -\frac{3}{5}, 0, 0), (0, 0, 0, 1), (0, 0, 1, 0), (0, 1, 0, 0), (\frac{6}{11}, \frac{6}{11}, \frac{7}{11}, 0), (\frac{4}{7}, -\frac{2}{7}, -\frac{2}{7}, \frac{5}{7}), (1, 0, 0, 0))$ \\[0.25em]
(7, 19, 23, 11) & $((0, -\frac{2}{7}, \frac{6}{7}, \frac{3}{7}), (0, 0, -1, 0), (0, 0, 0, 1), (0, 0, 1, 0), (0, 1, 0, 0), (\frac{1}{6}, \frac{5}{6}, -\frac{1}{6}, \frac{1}{2}), (1, 0, 0, 0))$ \\[0.25em]
(7, 19, 23, 11) & $((-\frac{2}{3}, 0, \frac{1}{3}, -\frac{2}{3}), (-\frac{3}{5}, 0, -\frac{4}{5}, 0), (0, 0, 0, 1), (0, 0, 1, 0), (0, 1, 0, 0), (\frac{6}{11}, \frac{6}{11}, \frac{7}{11}, 0), (1, 0, 0, 0))$ \\[0.25em]
(7, 19, 24, 12) & $((-\frac{2}{5}, \frac{4}{5}, -\frac{2}{5}, -\frac{1}{5}), (0, -\frac{4}{5}, -\frac{3}{5}, 0), (0, 0, 0, 1), (0, 0, 1, 0), (0, 1, 0, 0), (\frac{6}{11}, \frac{6}{11}, \frac{7}{11}, 0), (1, 0, 0, 0))$ \\[0.25em]
(7, 20, 25, 12) & $((-\frac{3}{5}, -\frac{4}{5}, 0, 0), (-\frac{2}{13}, \frac{10}{13}, \frac{4}{13}, \frac{7}{13}), (0, 0, 0, 1), (0, 0, 1, 0), (0, 1, 0, 0), (\frac{6}{11}, \frac{6}{11}, \frac{7}{11}, 0), (1, 0, 0, 0))$ \\[0.25em]
(7, 20, 26, 13) & $((-\frac{1}{6}, -\frac{1}{6}, \frac{5}{6}, \frac{1}{2}), (0, 0, 0, -1), (0, 0, 0, 1), (0, 0, 1, 0), (0, 1, 0, 0), (\frac{2}{3}, \frac{5}{9}, -\frac{2}{9}, \frac{4}{9}), (1, 0, 0, 0))$ \\[0.25em]
(7, 20, 26, 13) & $((-\frac{1}{2}, \frac{1}{2}, \frac{1}{2}, -\frac{1}{2}), (0, -\frac{4}{5}, -\frac{3}{5}, 0), (0, 0, 0, 1), (0, 0, 1, 0), (0, 1, 0, 0), (\frac{6}{11}, \frac{6}{11}, \frac{7}{11}, 0), (1, 0, 0, 0))$ \\[0.25em]
(7, 21, 28, 14) & $((-\frac{2}{13}, \frac{7}{13}, \frac{4}{13}, \frac{10}{13}), (0, 0, 0, -1), (0, 0, 0, 1), (0, 0, 1, 0), (0, 1, 0, 0), (\frac{5}{7}, \frac{4}{7}, -\frac{2}{7}, \frac{2}{7}), (1, 0, 0, 0))$ \\[1em]
(8, 18, 17, 7) & $((-\frac{4}{5}, -\frac{3}{5}, 0, 0), (-\frac{4}{5}, \frac{3}{5}, 0, 0), (0, -\frac{3}{5}, \frac{4}{5}, 0), (0, 0, 0, 1), (0, \frac{3}{5}, \frac{4}{5}, 0), (0, 1, 0, 0), (\frac{4}{5}, -\frac{3}{5}, 0, 0), (\frac{4}{5}, \frac{3}{5}, 0, 0))$ \\[0.25em]
(8, 18, 17, 7) & $((-\frac{4}{5}, 0, -\frac{3}{5}, 0), (-\frac{3}{5}, 0, 0, \frac{4}{5}), (-\frac{3}{5}, 0, \frac{4}{5}, 0), (-\frac{3}{5}, \frac{4}{5}, 0, 0), (\frac{3}{5}, 0, -\frac{4}{5}, 0), (\frac{3}{5}, 0, 0, \frac{4}{5}), (\frac{3}{5}, 0, \frac{4}{5}, 0), (\frac{3}{5}, \frac{4}{5}, 0, 0))$ \\[0.25em]
(8, 18, 17, 7) & $((0, -1, 1, 0), (0, 0, 1, -1), (0, 0, 1, 1), (0, 1, -1, 0), (0, 1, 0, -1), (0, 1, 0, 1), (0, 1, 1, 0), (1, 1, 0, 0))$ \\[0.25em]
(8, 18, 17, 7) & $((-\frac{1}{2}, -\frac{1}{2}, 0, 0), (-\frac{1}{2}, 0, 0, -\frac{1}{2}), (0, -\frac{1}{2}, -\frac{1}{2}, 0), (0, 0, -\frac{1}{2}, -\frac{1}{2}), (0, 0, 0, 1), (0, 0, 1, 0), (0, 1, 0, 0), (1, 0, 0, 0))$ \\[0.25em]
(8, 19, 18, 7) & $((-\frac{1}{2}, -\frac{1}{2}, -\frac{1}{2}, -\frac{1}{2}), (-\frac{1}{2}, -\frac{1}{2}, -\frac{1}{2}, \frac{1}{2}), (-\frac{1}{2}, -\frac{1}{2}, \frac{1}{2}, -\frac{1}{2}), (-\frac{1}{2}, -\frac{1}{2}, \frac{7}{10}, -\frac{1}{10}), (-\frac{1}{2}, \frac{1}{2}, -\frac{1}{2}, -\frac{1}{2}), (-\frac{1}{2}, \frac{1}{2}, \frac{1}{2}, -\frac{1}{2}),$ \\[0.25em]
               & $(-\frac{1}{10}, \frac{7}{10}, -\frac{1}{2}, -\frac{1}{2}), (\frac{1}{2}, -\frac{1}{2}, -\frac{1}{2}, -\frac{1}{2}))$ \\[1em]
(9, 19, 17, 7) & $((-\frac{3}{5}, -\frac{4}{5}, 0, 0), (-\frac{3}{5}, 0, 0, \frac{4}{5}), (-\frac{3}{5}, 0, \frac{4}{5}, 0), (-\frac{3}{5}, \frac{4}{5}, 0, 0), (\frac{3}{5}, -\frac{4}{5}, 0, 0), (\frac{3}{5}, 0, 0, \frac{4}{5}), (\frac{3}{5}, 0, \frac{4}{5}, 0), (\frac{3}{5}, \frac{4}{5}, 0, 0), (1, 0, 0, 0))$ \\[0.25em]
(9, 20, 18, 7) & $((0, -\frac{2}{3}, -\frac{2}{3}, -\frac{1}{3}), (0, -\frac{2}{3}, -\frac{2}{3}, \frac{1}{3}), (0, -\frac{2}{3}, \frac{2}{3}, -\frac{1}{3}), (0, -\frac{2}{3}, \frac{2}{3}, \frac{1}{3}), (0, \frac{2}{3}, -\frac{2}{3}, -\frac{1}{3}), (0, \frac{2}{3}, -\frac{2}{3}, \frac{1}{3}), (0, \frac{2}{3}, \frac{2}{3}, -\frac{1}{3}),$ \\[0.25em]
               & $(0, \frac{2}{3}, \frac{2}{3}, \frac{1}{3}), (1, 0, 0, 0))$ \\[0.25em] 
\end{tabular}}

\newpage
\section{Realization of $C_4(7)^*$}
\label{app:c47}

Let $a$ denote the root of the irreducible polynomial 
\begin{align*}
2000 \, x^{10} & - 61600 \, x^{8} + 84000 \, x^{7} + 550760 \, x^{6} - 1234800 \, x^{5} \\
               & - 2287712 \, x^{4} + 11660040 \, x^{3} - 17853395 \, x^{2} + 12862500 \, x - 3721550\in\Q[x],
\end{align*}
which is approximately equal to $0.9989495\dots$.
The point $p_{(ijkl)^*}$ is represented by a matrix where the $ij$-th entry represents the coefficient of $a^{j-1}$ in the $i$-th coordinate of $p_{(ijkl)^*}$, with the common denominator on the left-hand side.

\smallskip
\noindent
$20070439200\cdot p_{(1237)^*}$ = 
\[
\resizebox{\textwidth}{!}{$%
\left(\begin{array}{>{\raggedleft}p{0.17\textwidth}>{\raggedleft}p{0.17\textwidth}>{\raggedleft}p{0.17\textwidth}>{\raggedleft}p{0.17\textwidth}>{\raggedleft}p{0.17\textwidth}>{\raggedleft}p{0.17\textwidth}>{\raggedleft}p{0.17\textwidth}>{\raggedleft}p{0.17\textwidth}>{\raggedleft}p{0.17\textwidth}p{0.17\textwidth}<{\raggedleft}}
-71971814950 & 214819961160 & -239913086315 & 103001684898 & 11944786350 & -19361648040 & 204242500 & 1993811400 & -101845000 & -68508000 \\
102148616850 & -287478698780 & 314362827345 & -135231668824 & -12541988550 & 25218409020 & -960907500 & -2679443200 & 185385000 & 98254000 \\
-58751181300 & 169926267000 & -177108168090 & 67631242560 & 12031670700 & -12956743800 & -363237000 & 1315608000 & -44610000 & -43260000
\end{array}\right)$}
\]
$566773200\cdot p_{(1245)^*}$ = 
\[
\resizebox{\textwidth}{!}{$%
\left(\begin{array}{>{\raggedleft}p{0.17\textwidth}>{\raggedleft}p{0.17\textwidth}>{\raggedleft}p{0.17\textwidth}>{\raggedleft}p{0.17\textwidth}>{\raggedleft}p{0.17\textwidth}>{\raggedleft}p{0.17\textwidth}>{\raggedleft}p{0.17\textwidth}>{\raggedleft}p{0.17\textwidth}>{\raggedleft}p{0.17\textwidth}p{0.17\textwidth}<{\raggedleft}}
19328773730 & -51170823480 & 50521437505 & -17774272446 & -3945546150 & 3488404080 & 193196500 & -351787800 & 6545000 & 11016000 \\
-39521739390 & 109005389220 & -116217783255 & 47552262072 & 5924107350 & -8971113060 & 122104500 & 939069600 & -53025000 & -33162000 \\
13632192000 & -29671387480 & 23564352840 & -3847036844 & -3941687400 & 996066120 & 467796000 & -77529200 & -22260000 & -76000
\end{array}\right)$}
\]
$56010528\cdot p_{(1256)^*}$ = 
\[
\resizebox{\textwidth}{!}{$%
\left(\begin{array}{>{\raggedleft}p{0.17\textwidth}>{\raggedleft}p{0.17\textwidth}>{\raggedleft}p{0.17\textwidth}>{\raggedleft}p{0.17\textwidth}>{\raggedleft}p{0.17\textwidth}>{\raggedleft}p{0.17\textwidth}>{\raggedleft}p{0.17\textwidth}>{\raggedleft}p{0.17\textwidth}>{\raggedleft}p{0.17\textwidth}p{0.17\textwidth}<{\raggedleft}}
-121296462 & 291892412 & -332564127 & 145513732 & 23348010 & -27045060 & -989100 & 2623600 & -63000 & -82000 \\
-10659754 & -12709032 & 23499763 & -29109234 & 6678210 & 5315520 & -1192100 & -592200 & 77000 & 24000 \\
34691020 & -25418064 & 46999526 & -58218468 & 13356420 & 10631040 & -2384200 & -1184400 & 154000 & 48000
\end{array}\right)$}
\]
$2867205600\cdot p_{(1267)^*}$ = 
\[
\resizebox{\textwidth}{!}{$%
\left(\begin{array}{>{\raggedleft}p{0.17\textwidth}>{\raggedleft}p{0.17\textwidth}>{\raggedleft}p{0.17\textwidth}>{\raggedleft}p{0.17\textwidth}>{\raggedleft}p{0.17\textwidth}>{\raggedleft}p{0.17\textwidth}>{\raggedleft}p{0.17\textwidth}>{\raggedleft}p{0.17\textwidth}>{\raggedleft}p{0.17\textwidth}p{0.17\textwidth}<{\raggedleft}}
-19300421350 & 46282126310 & -43936222155 & 15479075353 & 2847331200 & -3041102190 & -50487500 & 322102900 & -13440000 & -10963000 \\
-12982515700 & 35506910670 & -36764273210 & 14405597171 & 2320167150 & -2785633830 & -57575000 & 284430300 & -9205000 & -9191000 \\
478725100 & -1272505500 & 6099370550 & -4474065050 & -295470000 & 807691500 & 31535000 & -68565000 & -1400000 & 1550000
\end{array}\right)$}
\]
$793482480\cdot p_{(1347)^*}$ = 
\[
\resizebox{\textwidth}{!}{$%
\left(\begin{array}{>{\raggedleft}p{0.17\textwidth}>{\raggedleft}p{0.17\textwidth}>{\raggedleft}p{0.17\textwidth}>{\raggedleft}p{0.17\textwidth}>{\raggedleft}p{0.17\textwidth}>{\raggedleft}p{0.17\textwidth}>{\raggedleft}p{0.17\textwidth}>{\raggedleft}p{0.17\textwidth}>{\raggedleft}p{0.17\textwidth}p{0.17\textwidth}<{\raggedleft}}
-19648634950 & 49266897120 & -46185749435 & 14344638966 & 4507541850 & -2918602680 & -338481500 & 284503800 & 4865000 & -7836000 \\
26282357850 & -73622222660 & 79646366205 & -33416324788 & -3668407050 & 6265642740 & -154045500 & -659688400 & 41055000 & 23698000 \\
566773200 & 0 & 0 & 0 & 0 & 0 & 0 & 0 & 0 & 0
\end{array}\right)$}
\]
$1983706200\cdot p_{(1457)^*}$ = 
\[
\resizebox{\textwidth}{!}{$%
\left(\begin{array}{>{\raggedleft}p{0.17\textwidth}>{\raggedleft}p{0.17\textwidth}>{\raggedleft}p{0.17\textwidth}>{\raggedleft}p{0.17\textwidth}>{\raggedleft}p{0.17\textwidth}>{\raggedleft}p{0.17\textwidth}>{\raggedleft}p{0.17\textwidth}>{\raggedleft}p{0.17\textwidth}>{\raggedleft}p{0.17\textwidth}p{0.17\textwidth}<{\raggedleft}}
-10224144000 & 22253540610 & -17673264630 & 2885277633 & 2956265550 & -747049590 & -350847000 & 58146900 & 16695000 & 57000 \\
-3408048000 & 7417846870 & -5891088210 & 961759211 & 985421850 & -249016530 & -116949000 & 19382300 & 5565000 & 19000 \\
18032093100 & -37089234350 & 29455441050 & -4808796055 & -4927109250 & 1245082650 & 584745000 & -96911500 & -27825000 & -95000
\end{array}\right)$}
\]
$7\cdot p_{(1567)^*}$ = 
\[
\resizebox{\textwidth}{!}{$%
\left(\begin{array}{>{\raggedleft}p{0.17\textwidth}>{\raggedleft}p{0.17\textwidth}>{\raggedleft}p{0.17\textwidth}>{\raggedleft}p{0.17\textwidth}>{\raggedleft}p{0.17\textwidth}>{\raggedleft}p{0.17\textwidth}>{\raggedleft}p{0.17\textwidth}>{\raggedleft}p{0.17\textwidth}>{\raggedleft}p{0.17\textwidth}p{0.17\textwidth}<{\raggedleft}}
-3 & 0 & 0 & 0 & 0 & 0 & 0 & 0 & 0 & 0 \\
-1 & 0 & 0 & 0 & 0 & 0 & 0 & 0 & 0 & 0 \\
5 & 0 & 0 & 0 & 0 & 0 & 0 & 0 & 0 & 0
\end{array}\right)$}
\]
$2867205600\cdot p_{(2345)^*}$ = 
\[
\resizebox{\textwidth}{!}{$%
\left(\begin{array}{>{\raggedleft}p{0.17\textwidth}>{\raggedleft}p{0.17\textwidth}>{\raggedleft}p{0.17\textwidth}>{\raggedleft}p{0.17\textwidth}>{\raggedleft}p{0.17\textwidth}>{\raggedleft}p{0.17\textwidth}>{\raggedleft}p{0.17\textwidth}>{\raggedleft}p{0.17\textwidth}>{\raggedleft}p{0.17\textwidth}p{0.17\textwidth}<{\raggedleft}}
728391215650 & -2094919908480 & 2255588894825 & -918271186602 & -127852684050 & 173832156960 & 331908500 & -17887938600 & 836395000 & 610392000 \\
-640177480950 & 1815290831340 & -2008474027875 & 880868550996 & 73047526650 & -163761989580 & 7313554500 & 17445682800 & -1256535000 & -644166000 \\
57014523300 & -204183568400 & 279428358930 & -144653044060 & -10489728900 & 26862838800 & -393351000 & -2657158000 & 106470000 & 84760000
\end{array}\right)$}
\]
$1720323360\cdot p_{(2356)^*}$ = 
\[
\resizebox{\textwidth}{!}{$%
\left(\begin{array}{>{\raggedleft}p{0.17\textwidth}>{\raggedleft}p{0.17\textwidth}>{\raggedleft}p{0.17\textwidth}>{\raggedleft}p{0.17\textwidth}>{\raggedleft}p{0.17\textwidth}>{\raggedleft}p{0.17\textwidth}>{\raggedleft}p{0.17\textwidth}>{\raggedleft}p{0.17\textwidth}>{\raggedleft}p{0.17\textwidth}p{0.17\textwidth}<{\raggedleft}}
11378013150 & -33812265760 & 37688384055 & -15738414566 & -1992004350 & 2921101680 & -30964500 & -302283800 & 17325000 & 10736000 \\
-9162336050 & 21572588100 & -19448165905 & 5712716940 & 1793730750 & -1154111700 & -108188500 & 117642000 & -1085000 & -3690000 \\
-16604348740 & 43145176200 & -38896331810 & 11425433880 & 3587461500 & -2308223400 & -216377000 & 235284000 & -2170000 & -7380000
\end{array}\right)$}
\]
$60211317600\cdot p_{(2367)^*}$ = 
\[
\resizebox{\textwidth}{!}{$%
\left(\begin{array}{>{\raggedleft}p{0.17\textwidth}>{\raggedleft}p{0.17\textwidth}>{\raggedleft}p{0.17\textwidth}>{\raggedleft}p{0.17\textwidth}>{\raggedleft}p{0.17\textwidth}>{\raggedleft}p{0.17\textwidth}>{\raggedleft}p{0.17\textwidth}>{\raggedleft}p{0.17\textwidth}>{\raggedleft}p{0.17\textwidth}p{0.17\textwidth}<{\raggedleft}}
275535072750 & -590111239480 & 409498429935 & -40933667138 & -62997386550 & 12663303240 & 6368575500 & -1213843400 & -235935000 & 16748000 \\
545270272750 & -1346220763860 & 1256883097295 & -397204485216 & -107019755850 & 78694530180 & 5548203500 & -8048728800 & 161455000 & 259986000 \\
-481579031500 & 1445819734800 & -1519614839750 & 586406775780 & 95634745500 & -110715284400 & -1234835000 & 11421354000 & -541450000 & -395880000
\end{array}\right)$}
\]
$4014087840\cdot p_{(3456)^*}$ = 
\[
\resizebox{\textwidth}{!}{$%
\left(\begin{array}{>{\raggedleft}p{0.17\textwidth}>{\raggedleft}p{0.17\textwidth}>{\raggedleft}p{0.17\textwidth}>{\raggedleft}p{0.17\textwidth}>{\raggedleft}p{0.17\textwidth}>{\raggedleft}p{0.17\textwidth}>{\raggedleft}p{0.17\textwidth}>{\raggedleft}p{0.17\textwidth}>{\raggedleft}p{0.17\textwidth}p{0.17\textwidth}<{\raggedleft}}
11154335990 & -21423180730 & 13733960835 & -1699549859 & -1024781100 & 430167570 & -17328500 & -68488700 & 8610000 & 3389000 \\
12886770680 & -35252409570 & 35544399100 & -13510784161 & -2261073150 & 2624095530 & 51268000 & -270717300 & 9485000 & 8881000 \\
29787629200 & -70504819140 & 71088798200 & -27021568322 & -4522146300 & 5248191060 & 102536000 & -541434600 & 18970000 & 17762000
\end{array}\right)$}
\]
$40007520\cdot p_{(3467)^*}$ = 
\[
\resizebox{\textwidth}{!}{$%
\left(\begin{array}{>{\raggedleft}p{0.17\textwidth}>{\raggedleft}p{0.17\textwidth}>{\raggedleft}p{0.17\textwidth}>{\raggedleft}p{0.17\textwidth}>{\raggedleft}p{0.17\textwidth}>{\raggedleft}p{0.17\textwidth}>{\raggedleft}p{0.17\textwidth}>{\raggedleft}p{0.17\textwidth}>{\raggedleft}p{0.17\textwidth}p{0.17\textwidth}<{\raggedleft}}
36683850 & -87262924 & 89481693 & -18404792 & -15459990 & 3525060 & 1656900 & -257600 & -63000 & 2000 \\
-9038050 & 49488432 & -66140249 & 49303506 & -4364430 & -9061080 & 1168300 & 961800 & -91000 & -36000 \\
98489020 & -183897000 & 213202430 & -100971360 & -11568900 & 18727800 & 119000 & -1848000 & 70000 & 60000
\end{array}\right)$}
\]
$p_{(4567)^*}$ = 
\[
\resizebox{\textwidth}{!}{$%
\left(\begin{array}{>{\raggedleft}p{0.17\textwidth}>{\raggedleft}p{0.17\textwidth}>{\raggedleft}p{0.17\textwidth}>{\raggedleft}p{0.17\textwidth}>{\raggedleft}p{0.17\textwidth}>{\raggedleft}p{0.17\textwidth}>{\raggedleft}p{0.17\textwidth}>{\raggedleft}p{0.17\textwidth}>{\raggedleft}p{0.17\textwidth}p{0.17\textwidth}<{\raggedleft}}
0 & 0 & 0 & 0 & 0 & 0 & 0 & 0 & 0 & 0 \\
0 & 0 & 0 & 0 & 0 & 0 & 0 & 0 & 0 & 0 \\
1 & 0 & 0 & 0 & 0 & 0 & 0 & 0 & 0 & 0
\end{array}\right)$}
\]
\end{document}